\documentclass[a4paper]{amsart}
\usepackage[foot]{amsaddr}

\usepackage{geometry}
\newcommand{\geometrylen}[1]{\csname Gm@#1\endcsname}

\usepackage{xcolor} \usepackage{algorithmic,algorithm} \usepackage[utf8]{inputenc}
\usepackage[T1]{fontenc}
\usepackage{siunitx}

\usepackage[hidelinks]{hyperref}
\usepackage[utf8]{inputenc}
\usepackage{amsmath, amsfonts,amssymb,amsthm}
\usepackage{graphicx}
\usepackage{stmaryrd}
\usepackage{caption}
\usepackage[position=t]{subcaption}

\usepackage[color=black!10!white, disable]{todonotes}
\usepackage{mathtools}
\usepackage{tikz}
\usetikzlibrary{colorbrewer, angles, calc, quotes,decorations,patterns}
\usepackage{pgfplots}
\usepackage{pgfplotstable}
\usepackage{booktabs}
\usepackage{placeins}
\usepackage{csquotes}

\usepgfplotslibrary{colorbrewer}
\pgfplotsset{compat=1.13,
  axis lines=left,
legend style={draw=none},
scaled ticks=false
}
\definecolor{maincolor}{gray}{0}

\usepackage{tgpagella}
\usepackage{microtype}
\usepackage[shortlabels]{enumitem}

\newcommand{\jump}[1]{\llbracket #1 \rrbracket}
\newcommand{\avg}[1]{\lbrace\!\!\lbrace #1 \rbrace\!\!\rbrace}
\newcommand{\dgnorm}[1]{\lVert #1 \rVert_\mathrm{DG}}
\newcommand{\fullnorm}[2][]{\left\vert\kern-0.25ex\left\vert\kern-0.25ex\left\vert #2 
      \right\vert\kern-0.25ex\right\vert\kern-0.25ex\right\vert_{\mathrm{DG} #1}}
\newcommand{\fullnormfixedsize}[2][]{\vert\kern-0.25ex\vert\kern-0.25ex\vert #2 
      \vert\kern-0.25ex\vert\kern-0.25ex\vert_{\mathrm{DG} #1}}

\newtheorem{theorem}{Theorem}

\newtheorem{lemma}{Lemma}
\newtheorem{remark}{Remark}

\newtheorem{ind-assumption}{Induction Assumption}

\newtheorem{corollary}[lemma]{Corollary}

\DeclareMathOperator{\spn}{span}

\DeclareMathOperator{\ind}{ind}
\DeclareMathOperator{\im}{im}

\let\epsilon\varepsilon

\let\phi\varphi

\let\theta\vartheta

\let\myempty\varnothing

\def\hp{$hp$}
\newcommand{\de}[1]{#1_\delta}
\newcommand{\spacenameinnorm}{\mathcal{D}}
\newcommand{\spacenameinnormhom}{\mathcal{G}}
\newcommand{\wnorm}[6]{\left\| #1 \right\|_{\mathcal{K}^{#2,#3}_{#4}(#6)} }

\newcommand{\wseminorm}[5]{\left| #1 \right|_{\mathcal{K}^{#2,#3}_{#4}(#5)} }
\newcommand{\wnormdg}[5]{\left\vert\kern-0.25ex\left\vert\kern-0.25ex\left\vert #1 
      \right\vert\kern-0.25ex\right\vert\kern-0.25ex\right\vert_{\spacenameinnorm^{#3}_{#4}(#5)} }
\newcommand{\wnormdghom}[5]{\left\vert\kern-0.25ex\left\vert\kern-0.25ex\left\vert #1 
      \right\vert\kern-0.25ex\right\vert\kern-0.25ex\right\vert_{\spacenameinnormhom^{#3}_{#4}(#5)} }
\newcommand{\dhe}{\mathtt{h_e}}
\newcommand{\dpe}{\mathtt{p_e}}

\newcommand{\slope}{\mathfrak{s}}

\newcommand{\ad}{a_{\delta}}

\newcommand{\ud}{{u_\delta}}

\newcommand{\hK}{{\hat{K}}}

\newcommand{\fc}{{\mathfrak{c}}}

\newcommand{\dOmega}{{\partial\Omega}}

\newcommand{\vd}{{v_\delta}}

\newcommand{\ld}{{\lambda_\delta}}

\newcommand{\Xd}{{X_\delta}}

\newcommand{\XXd}{{X(\delta)}}
\newcommand{\Td}{{T_\delta}}
\newcommand{\Ed}{{E_\delta}}
\newcommand{\Ld}{{\Lambda_\delta}}
\newcommand{\Thd}{{\widehat{T}_{\delta}}}
\newcommand{\Th}{{\widehat{T}}}
\newcommand{\hd}{{\hat{\delta}}}

\newcommand{\mdj}{{{\mu}_{\delta j}}}
\newcommand{\ldj}{{{\lambda}_{\delta j}}}
\newcommand{\udj}{{{u}_{\delta j}}}
\newcommand{\xd}{{x_\delta}}

\newcommand{\KsgO}{{\mathcal{K}^s_\gamma(\Omega)}}
\newcommand{\JsgO}{{\mathcal{J}^s_\gamma(\Omega)}}

\renewcommand{\Re}{\operatorname{Re}}

\renewcommand{\emptyset}{\varnothing}

\graphicspath{{./figures/}}
\pgfplotstableset{    columns/N/.style={int detect,column type=r,column name=$N$},
    columns/eDG/.style={
	sci sep align, precision=2,
	column name=$\|u-u_\delta\|_{\mathrm{DG}}$,
      },
    columns/eL2/.style={
	sci sep align, precision=2,
	column name=$\|u-u_\delta\|_{L^2(\Omega)}$,
      },
    columns/eLinf/.style={
	sci sep align, precision=2,
	column name=$\|u-u_\delta\|_{L^\infty(\Omega)}$,
    },
    columns/elambda/.style={
	sci sep align, precision=2,
        column name=$|\lambda - \lambda_\delta|$,
      },
    columns/bL2/.style={
	sci sep align, precision=2,
	column name=$b_{L^2}$,
      },
    columns/bDG/.style={
	sci sep align, precision=2,
	column name=$b_{\mathrm{DG}}$,
    },
    columns/bLinf/.style={
	sci sep align, precision=2,
	column name=$b_{L^\infty}$,
    },
    columns/blambda/.style={
	sci sep align, precision=2,
	column name=$b_{\lambda}$,
    },
    columns/slope/.style={
	sci sep align, precision=4,
	column name=$\mathfrak{s}$,
    },
    empty cells with={--},     every head row/.style={before row=\toprule,after row=\midrule},
    every last row/.style={after row=\bottomrule}
  }

\title[Regularity and hp dG FE approximation of linear elliptic problems with
singular potentials]{Regularity and $hp$ discontinuous Galerkin finite element approximation of linear
  elliptic eigenvalue problems with singular potentials}
\author{Yvon Maday$^{\dagger, \star}$ }
\author{Carlo Marcati$^\dagger$}
\address{$^\dagger$ Sorbonne Université, Université Paris-Diderot SPC, CNRS, Laboratoire
  Jacques-Louis Lions, LJLL, F-75005 Paris \\
  $^\star$ Institut Universitaire de France 
   }
 \email{yvon.maday@upmc.fr, carlo.marcati@upmc.fr}
\begin{document}

\begin{abstract}
  We study the regularity in weighted Sobolev spaces of
  Schr\"{o}dinger-type eigenvalue problems, and we analyse their approximation
  via a discontinuous Galerkin (dG) $hp$ finite element method.
  In particular, we show that, for a class of singular potentials, the
  eigenfunctions of the operator belong to analytic-type non homogeneous
  weighted Sobolev spaces. Using this result, we prove that the an isotropically
  graded $hp$ dG method is spectrally accurate, and that the numerical
  approximation converges with exponential rate to the exact solution. Numerical
  tests in two and three dimensions confirm the theoretical results and provide
  an insight into the behaviour of the method for varying discretisation parameters.
\end{abstract}
\keywords{$hp$ graded finite element method, discontinuous Galerkin, elliptic eigenvalue problem,
Schr\"{o}dinger equation, weighted Sobolev spaces, elliptic regularity}

\subjclass[2010]{ 35J10, 65N25, 65N30}

\maketitle

\section{Introduction}
Many problems in physics and chemistry are modeled through elliptic eigenvalue
problems with singular potential. This is the case, for example, of the
electronic Schr\"{o}dinger equation, where the attraction between the nuclei and
the electrons is proportional to the inverse of their distance.
In this paper we propose and analyze the application of an isotropically graded
\hp{} discontinuous Galerkin (dG) finite element method for the approximation of the
solution to linear elliptic eigenvalue problems. The central idea is that, for a
wide class of singular potentials, the exact eigenfunctions are highly regular in
weighted Sobolev spaces -- i.e., they are smooth in Sobolev spaces when
multiplied by weights that are null at the singularities. The weighted Sobolev
spaces considered were introduced in the analysis of
elliptic problems in domains with non-smooth boundary \cite{Kondratev1967}; when
applied to elliptic
problems in domains with corners and edges, the graded \hp{} refinement gives rise to exponentially
convergent methods \cite{Guo1986a, Guo1986b, Schotzau2013a, Schotzau2013b}.

Our goal is firstly, therefore, to show that the solution to the eigenvalue problems
has sufficient regularity to be approximated with exponential convergence by the
discontinuous \hp{} space. Then, this can be used to prove that the solution
provided by the \hp{} dG finite element method converges with this
exponential rate.

In Section \ref{sec:notation-statement} we start by briefly introducing the
functional setting of homogeneous and non homogeneous weighted Sobolev spaces,
and by stating our eigenvalue problem. We do so in a quite general way, which
includes both singularities on the boundary and in the interior of the domain.
Note however that in three dimensions we do not consider anisotropic
approximation along the edge, thus the singularities only arise in practice
from potentials.

In Section \ref{sec:regularity} we consider the issues related to the regularity of solutions to linear elliptic problem with singular points. We are mainly interested in singular points as a consequence of singular potentials, but we place ourselves in the more general case of a conical domain. The analysis therefore applies also to corner domains in two and three dimensions, a situation that has been widely studied, see, among the others, \cite{Costabel2012, Egorov1997, Kozlov1997,Mazya2010}.

Let us consider a conical domain, i.e., a bounded domain $\Omega \subset
\mathbb{R}^d$ such that, after localization of the singularity at the origin, in
polar coordinates,
$\Omega\cap S_{d-1} = (0, \zeta ) \times U $,
with $\zeta > 0$, $U\subset \mathbb{S}_{d-1}$, and $S_{d-1}$ is the $d-1$ dimensionsal sphere, .
While most of the literature is concerned with the analysis in homogeneous
weighted Sobolev spaces, denoted here as $\KsgO$, here we focus on inhomogeneous
spaces, denoted as $\JsgO$. The latter
spaces have been studied mainly as the domain of solutions to elliptic problems
in corner domains with Neumann boundary conditions. The similarity arises from
the fact that problems with singular potential and Neumann boundary problems in
domains with conical points share solutions that a priori, have nonzero imposed
value at the singular points. 

  The reason why a regularity result in non homogeneous weighted spaces is more
  relevant than its homogeneous counterpart lies in the fact that, by taking
  wider spaces --- in general, $\mathcal{K}^{s,p}_\gamma
  (\Omega)\subset\mathcal{J}^{s,p}_\gamma (\Omega)$ --- we can obtain an
  estimate with a bigger weight $\gamma$.
  This is relevant since it can, in some situations, give insight into the
  boundedness of a function.

From the point of view of the Mellin transformation, working in non homogeneous
spaces consists in isolating some singularities of the Mellin transform of the
solution, bounding the rest of the function using the theory of homogeneous
spaces, and finally bounding the terms in the expansion of the solution
corresponding to the singularities via embeddings in higher order non homogeneous
spaces. To illustrate this, consider the Mellin symbol $\mathfrak{L}$
related to a Laplacian operator in a conical domain given by $\mathfrak{L}
(\omega, \partial_\omega, \lambda) = -\left( (\lambda+d-2)\lambda + \Delta_U
\right)$, with $-\Delta_U$ representing the Laplace-Beltrami operator on
$U\subset \mathbb{S}_{d-1}$, in the case where $\Delta_U$ has a null eigenvalue (corresponding to spherically symmetric functions). The symbol has a single (resp. double) zero for $\lambda = 0$ in three (resp. two) dimensions. In three dimension, this zero corresponds to a constant in the asymptotic expansion of the solution near the singularity; in two dimensions, we would have a constant and logarithmic term $\log(|x|)$, but this would not be in $H^1(\Omega)$. In the asymptotic expansion of the solution near the singularity, we will therefore find a constant, followed by a term due to the potential or the geometry of the domain.
The former case depends on the asymptotic expansion of the potential near the singularity, while the latter depends on the eigenvalues of $\Delta_U$. In the following sections, we will suppose that the term following the constant in the expansion goes as $|x|^\epsilon$ for an $0 <\epsilon< 1$. 
As an example of a potential that would generate such a behavior, consider
$V(x) = |x|^{-2+\epsilon}$. A geometry causing an expansion containing
$|x|^\epsilon$ would instead be one such that 
$\epsilon(\epsilon+d-2) \in \sigma(-\Delta_U, B_{\partial U})$,
i.e., there exists a function $\hat{u}:\mathbb{R}^+ \to H^s(U)$, $s\geq 2$ such that
\begin{equation*}
   \left(\mathfrak{L} \hat{u}\right)(\lambda) = -\left((\lambda+d-2)\lambda - (\epsilon+d-2)\epsilon\right) \hat{u}(\lambda).
\end{equation*}
As it can easily be seen, $\epsilon$ is indeed a zero of the symbol above. More
practically, this happens if we consider a two dimensional wedge with aperture
$\pi/\epsilon$ (in this two dimensional wedge case we have therefore also $\epsilon\geq 1/2$), as it will be outlined later in Section \ref{sec:problem}.

Returning to weighted Sobolev spaces, in light of the analysis of the operator
given above, we can consider a simple case by neglecting the higher order terms,
and consider a function $v(x) = v(0) + |x|^\epsilon$, with $v(0)\neq 0$ as a
model of our solution. As long as $|x|\ll 1$, those terms are indeed the
predominant ones. The norm $\|r^{-d/p} v \|_{L^p(\Omega)}$ is clearly unbounded,
thus $v\notin \mathcal{K}^{s,p}_{\gamma}(\Omega)$ for any $s\in \mathbb{N}$ and
any $\gamma \geq d/p$. It is easy to see, though, that the statement $v\in
\mathcal{K}^{s,p}_{\gamma}(\Omega)$ for any $s\in\mathbb{N}$ and $\gamma< d/p$
does not tell the whole story, since $v-v(0)\in\mathcal{K}^{s,
  p}_{\gamma}(\Omega)$ also for larger values of $\gamma \in (d/p, d/p+\epsilon)$. The non homogeneous weighted spaces give therefore a framework where functions such as $v$ can be treated more naturally than in homogeneous spaces.

We define the spaces treated above in more detail and outline the relationships between the homogeneous and non homogeneous ones in the following Section \ref{sec:weighted}. Then, in Section \ref{sec:problem} we specify the class of operators we treat here. The main regularity result for those operators is then given in Section \ref{sec:regularity}. Specifically, we give an elliptic regularity result in non homogeneous weighted Sobolev spaces for operators with singular potential, and we follow with an observation on how this can be used as a basis to obtain ``analytic regularity'' in weighted spaces -- see Corollary \ref{cor:analytic}.

In Section \ref{sec:numerical-appx} we introduce the \hp{} discontinuous
Galerkin method we use to approximate the solution to the eigenproblems
considered and we prove our convergence results.

Historically, dG methods have been originally introduced for the approximation of first order
steady equations in \cite{Reed1973, Crouzeix1973} in the context of neutron transport
equations and of the Stokes equation. For second order elliptic problems, the development of discontinuous Galerkin methods is based on the ideas in \cite{Nitsche1972}, with interior penalty methods
being introduced in \cite{Wheeler1978} and developed in \cite{Arnold1982}. A wide range of different
methods have been proposed throughout the years, including, among others, the
local discontinuous Galerkin (LDG) method \cite{Cockburn1998}, and the already
mentioned class of interior penalty (IP) methods, in its symmetric (SIP), nonsymmetric
(NIP) and incomplete (IIP) versions. See \cite{Riviere2008, Hesthaven2008, DiPietro2011} for an overview of discontinuous Galerkin methods.

The \hp{} version of finite element (FE) methods, introduced in \cite{Gui1986a,
  Gui1986b, Gui1986c} in one dimension and in \cite{Guo1986a, Guo1986b} in more
dimensions, combines adaptivity in space in low regularity regions with
adaptivity in polynomial degree in high regularity regions. When applied to
elliptic problems with point singularities, the numerical solutions obtained
with the \hp{} FE method can converge with exponential rate, provided that the
exact solutions belong to the spaces $\mathcal{K}^\varpi_\gamma(\Omega)$ or
$\mathcal{J}^\varpi_\gamma(\Omega)$ defined in Section \ref{sec:weighted}. We also signal the recent research on \hp{} methods in polygonal and polyhedral domains, see, among others, \cite{Costabel2005,
  Stamm2010, Schotzau2013a, Schotzau2013b,
   Schotzau2016}.

 We focus on the symmetric version of the interior penalty method, since the
 original problem is itself symmetric and preserving symmetry improves both the
 numerical stability and the theoretical convergence rate of the method. In
 particular, as shown in Theorem \ref{th:convergence}, the eigenvalues can be
 shown to converge at a rate twice that of the eigenfunctions. Our convergence
 analysis follows closely what has been shown in \cite{Antonietti2006}, with
 some minor differences related to the specificity of the regularity of the
 eigenspaces and to the approximation of the \hp{} space.

 We conclude, in Section \ref{sec:numerical-results}, with some numerical tests
 in two and three dimensions, in which we confirm our theoretical results and
 investigate the role of the sources of numerical error that we did not consider
 in the analysis. Those tests can also be used, in practice, to devise
 specifically crafted \hp{} spaces for more complex problems.

\section{Notation and statement of the problem} 
\label{sec:notation-statement}
Let us consider a bounded domain $\Omega\subset \mathbb{R}^d$, $d=2,3$ which will be specified later and let $\mathcal{C}$ be a set of isolated points in $\Omega$; for the sake of simplicity we consider the case of a single point $\mathcal{C} = \{\fc\}$; the theory can be trivially extended to the case of a finite number of points. We then denote by $r = r(x)$ the distance $|x-\mathfrak{c}|$ where $|\cdot|$ is the euclidean norm of $\mathbb{R}^d$ (it would be a smooth function representing the distance from the nearest point in $\mathcal{C}$ if there were more than one).

We denote by $C\in \mathbb{R}^+$ a generic constant independent of the discretization, and write $A \lesssim B$ (resp. $A\gtrsim B$) if $A\leq C B$ (resp. $A \geq C B $) and $A\simeq B$ if both $A\lesssim B$ and $A\gtrsim B$ hold.
\subsection{Weighted Sobolev spaces}
\label{sec:weighted}
For $k\in \mathbb{N}$ and $1 \leq p < \infty$e introduce on a set $D\subset \Omega$ the homogeneous weighted norm
\begin{equation}
  \label{eq:weighted-norm}
      \wnorm{u}{k}{p}{\gamma}{\beta}{D}^p  = \sum_{j=0}^k \sum_{|\alpha|=j}\|r^{j-\gamma}\partial^\alpha u\|^p_{L^{p}(D)} 
\end{equation}
with seminorm
\begin{equation*}
  \wseminorm{u}{k}{p}{\gamma}{D}^p  =  \sum_{|\alpha|=k}\|r^{k-\gamma}\partial^\alpha u\|^p_{L^{p}(D)} 
\end{equation*}
and denote the space $\mathcal{K}^{k,p}_{\gamma}(\Omega)$ as the space of all
functions with bounded $\mathcal{K}^{k,p}_{\gamma}(\Omega)$ norm. The case where
$p=\infty$ follows from the usual modification in Sobolev spaces.
We also introduce the inhomogeneous norm
\begin{equation}
  \label{eq:weighted-nonhom-norm}  
  \|u\|^p_{\mathcal{J}^{k,p}_{\gamma}(\Omega)} = \sum_{j=0}^k \sum_{|\alpha|=j} \|r^{\max(-\gamma+|\alpha|, \rho)} \partial^\alpha u\|^p_{L^p(\Omega)},
\end{equation}
for $\gamma-d/p<k$ and $\rho\in (-d/p, -\gamma+k]$, if $1\leq p < \infty$,
$\rho\in [0, -\gamma+k]$ if $p=\infty$. We write $\mathcal{K}^{k}_\gamma
(\Omega)= \mathcal{K}^{k,2}_\gamma(\Omega)$ and
$\mathcal{J}^{k}_\gamma=\mathcal{J}^{k,2}_\gamma$. We also remark that, for
$k\geq 1$ and $\gamma -d/2<0$, $\mathcal{K}^k_\gamma(\Omega) =
\mathcal{J}^k_\gamma(\Omega)$. Furthermore,  if $v\in
\mathcal{J}^k_\gamma(\Omega)$ for $k \geq 1 $ and $0<\gamma-d/2<1$
(condition under which $|v(c)| \lesssim \| v\|_{\mathcal{J}_\gamma^k(\Omega)}$),
\begin{equation}
  \label{eq:equivalence-hom-nonhom}
  \|v-v(\mathfrak{c})\|_{\mathcal{K}^k_\gamma(\Omega)} +|v(\mathfrak{c})| \simeq \|v\|_{\mathcal{J}^k_\gamma(\Omega)}.
\end{equation}
On the boundary, for integer $k\geq 1$, we introduce the space $\mathcal{K}^{k-1/p,p}_{\gamma-1/p}(\partial\Omega)$ (resp. $\mathcal{J}^{k-1/p,p}_{\gamma-1/p}(\partial\Omega)$) of traces of functions from $\mathcal{K}^{k,p}_{\gamma}(\Omega)$ (resp. $\mathcal{J}^{k,p}_{\gamma}(\Omega)$) with norm
\begin{equation*}
  \| u\|_{\mathcal{K}^{k-1/p,p}_{\gamma-1/p}(\partial\Omega)} = \inf\{\|v\|_{\mathcal{K}^{k,p}_{\gamma}(\Omega)}, v_{|_{\partial \Omega}} = u \},
\end{equation*}
and
\begin{equation*}
  \| u\|_{\mathcal{J}^{k-1/p,p}_{\gamma-1/p}(\partial\Omega)} = \inf\{\|v\|_{\mathcal{J}^{k,p}_{\gamma}(\Omega)}, v_{|_{\partial \Omega}} = u \}.
\end{equation*}
Note that on portions of the boundary not touching the singularity $\fc$, the
weighted trace spaces coincide with classical Sobolev trace spaces.

Finally, we introduce the spaces 
\begin{equation*}
    \mathcal{J}_\gamma^{\varpi,p}(\Omega) 
= \left\{v\in \mathcal{J}^{\infty,p}_\gamma(\Omega): \exists  A,C\in \mathbb{R}\text{ s.t. }\forall k\in\mathbb{N}, \,\|v\|_{\mathcal{J}^{k,p}_\gamma(\Omega)}\leq CA^kk!  \right\},
\end{equation*}
and 
\begin{equation*}
  \mathcal{K}_\gamma^{\varpi,p}(\Omega) 
= \left\{v\in \mathcal{K}^{\infty,p}_\gamma(\Omega): \exists  A,C\in \mathbb{R}\text{ s.t. }\forall k\in\mathbb{N},\,|v|_{\mathcal{K}^{k,p}_\gamma(\Omega)}\leq CA^kk! \right\}.
\end{equation*}
 In the following, for an $S\subset \Omega$, we denote by $(\cdot, \cdot)_S$ the $L^2(S)$ scalar product and by $\| \cdot \|_S$ the $L^2(S)$ norm. 
  \subsection{Statement of the problem}
  \label{sec:problem}
  Let us now assume that in a neighborhood of $\fc$,
  the domain $\Omega\subset \mathbb{R}^d$ is conical, i.e., there exists a ball $B_\zeta(\fc)$ centered in
  $\fc$ with radius $\zeta>0$ such that, going from a cartesian to a polar representation,
  \begin{equation*}
    \Omega\cap B_\zeta(\fc) = (0,\zeta)\times U
  \end{equation*}
  where $U\subset \mathbb{S}_{d-1}$ the $d-1$ dimensional sphere, and $\partial U $ is smooth.
              
  In this domain we set the problem of finding $\lambda \in \mathbb{R}$ and $u\in H^1(\Omega)$ such that $\| u \| = 1$ and
  \begin{equation}
    \label{eq:mainproblem} 
    \begin{aligned}
      L(x, \partial_x) u &= -\Delta u + V (x) u = \lambda u &&\text{ in } \Omega \\
      B(x, \partial_x) u &= 0&&\text{ on } \partial\Omega
    \end{aligned}
  \end{equation}
  where 
    $B(x, \partial_x)$ is a boundary operator with analytic coefficients of order $m\leq 1$ covering $L(x,\partial_x)$, i.e., such as the problem defined by $(L,B)$ is elliptic. 
Furthermore, $V:\mathbb{R}^d\to\mathbb{R}_+$ is a potential such that
$V\in\mathcal{K}^{\varpi,\infty}_{\epsilon-2}(\Omega)\cap L^p(\Omega)$ for some $0<\epsilon\leq
1$ and $p>d/2$, and $V$ is bounded from below by a positive constant.
We will omit the dependence of $L$ and $B$ on $x$ and $\partial_x$ when it will not be strictly necessary.
Finally, we denote by $a(\cdot, \cdot)$ be the bilinear form associated to $L$, i.e.
\begin{equation*}
  a(u,v) = (\nabla u, \nabla v) + (V u, v).
\end{equation*}

We recall the definition of the Mellin transformation 
\begin{equation*}
  \hat{u}(\lambda) = \left(\mathcal{M}_{r\to \lambda}\right)u = \int_{\mathbb{R}_+} u(r, \omega) r^{-\lambda-1} dr
\end{equation*}
where $(r,\omega)\in \mathbb{R}_+\times \mathbb{S}_{d-1}$ are spherical coordinates. 
We denote the leading part of the operator $L$ in \eqref{eq:mainproblem} by $L^0
= -\Delta$ and introduce the Mellin symbol $\mathfrak{L}(\omega, \partial_\omega, \lambda)$
of the leading part $L^0$, such as
\begin{equation}
  \label{eq:mellin-symbol}
  r^{-2} \mathfrak{L}(\omega, \partial_\omega, r\partial_r)  =  L^0(x, \partial_x),
\end{equation}
i.e., $\mathfrak{L} (\omega, \partial_\omega, \lambda) = -\left( (\lambda+d-2)\lambda + \Delta_U \right)$.

We also suppose, for ease of notation, that the smallest nonzero eigenvalue of
  the Laplace-Beltrami operator $-\Delta_U$ on $U$  with boundary operator
  $B_{\partial U}$ on
  $\partial U$ is bigger than $\epsilon(\epsilon+d-2)$, i.e.,
  \begin{equation}
    \label{eq:condition-beltrami}
    \min\left\{ \mu>0: \mu\in \sigma( -\Delta_U, B_{\partial U}) \right\}
    \geq \epsilon(\epsilon+d-2).
  \end{equation}
This condition, combined with
$V\in\mathcal{K}^{\varpi,\infty}_{\epsilon-2}(\Omega)$ guarantees that the
positive pole with smallest real part of the Mellin transform of the solution
$\hat{u}(\lambda)$ lies in the half-space $\{ \Re(\lambda) \geq \epsilon\}$, see
\cite[Chapter 6]{Kozlov1997}.
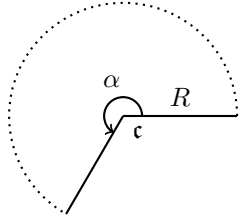
\begin{figure}
  \centering
  \begin{tikzpicture}
    \def\zoom{1.5}
    \coordinate (O) at (0,0)  {};
    \node[below right] at (O) {$\fc$};
    \coordinate (A) at (\zoom,0)  {};
    \coordinate (B) at (-\zoom*.5,{-\zoom*sqrt(3)/2})  {};
    \draw[thick] (A) -- (O) node[pos=.5, above] {$R$};
    \draw[thick] (B) -- (O);
    \pic["$\alpha$"{anchor=south},draw,->,thick,angle radius=.25cm,angle eccentricity=1.2] {angle = A--O--B};
    \pic [draw,dotted,thick,angle radius=\zoom cm, angle eccentricity=1] {angle = A--O--B};
  \end{tikzpicture}
  \caption{Two dimensional wedge}
  \label{fig:wedge}
\end{figure}
  As an example of condition \eqref{eq:condition-beltrami}, consider a two dimensional domain that coincides near the origin with the wedge with angle of aperture $\alpha\in(0,2\pi)$ \[\{0<r<R, \theta\in(0,\alpha)\},\] 
  where $r$ and $\theta$ are polar coordinates, as in Figure \ref{fig:wedge}. On the boundary we impose either homogeneous Dirichlet or
  homogeneous Neumann boundary conditions. Then,
  \eqref{eq:condition-beltrami} is equivalent to 
  \begin{equation*}
    \alpha \leq \frac{\pi}{\epsilon}.
  \end{equation*}
\section{Regularity of the solution}
\label{sec:regularity}
The first lemma concerns the regularity of the solution of \eqref{eq:mainproblem}: we specialize here the results of \cite{Kozlov1997}.
  We also introduce the set $I_d$ as 
  \begin{equation*}
    I_d =
    \begin{cases}
      (-1, \epsilon)\setminus\{0\} & \text{ if }d=3\\
      [0, \epsilon)& \text{ if }d=2.\\
    \end{cases}
  \end{equation*}
  In what follows, we analyze the set of weighted spaces where the operator
  $(L,B)$ is an isomorphism.
      We place ourselves in the Hilbertian setting ($p=2$). In general, we avoid
  considering the cases where $\gamma-d/2\in \mathbb{N}$, since for those
  $\gamma$ the operator is not Fredholm, the exception being $\gamma = 1$ when
  $d=2$, since $\mathcal{J}^1_1(\Omega) = H^1(\Omega)$.
  Given $\gamma\in I_d$, the image of the operator $(L,B)$ applied to $\mathcal{J}^k_\gamma(\Omega)$
  is given by $\mathcal{J}^{k-2}_{\gamma-2}(\Omega)\times
  \mathcal{J}^{k-m-1/2}_{\gamma-m-1/2}(\partial \Omega) $. In the following
 lemma we will show that the operator is an isomorphism between those spaces.

  The idea of the proof is then to start from in homogeneous weighted spaces
  spaces and then to extend the results to the non homogeneous ones, by function decomposition.
\begin{lemma}
  \label{lemma:solution-regularity}
  The operator $(L,B)$ is an isomorphism  between the spaces
    \begin{equation}
      \label{eq:spaces}
      \mathcal{J}^{k}_{\gamma}(\Omega) \to \mathcal{J}^{k-2}_{\gamma-2}(\Omega)\times \mathcal{J}^{k-m-1/2}_{\gamma-m-1/2}(\partial \Omega) 
    \end{equation}
    for $\gamma-d/2\in I_d$, 
{$k\geq 2$}.
                              \end{lemma}
  \begin{proof}
    Let $ \mathcal{F}_\gamma= (L_\gamma,B_\gamma) : \mathcal{K}^2_\gamma (\Omega)\to \mathcal{K}^{0}_{\gamma-2}(\Omega)\times \mathcal{K}^{3/2-m}_{\gamma-m-1/2}(\partial \Omega)$.
                    The operator $\mathcal{F}_\gamma$ is Fredholm for all $\gamma-d/2\notin \mathbb{N}$ \cite{Kozlov1997}; its index is defined as 
    \begin{equation*}
      \ind \mathcal{F}_\gamma = \dim(\ker \mathcal{F}_\gamma) - \dim (\ker \mathcal{F}^*_{\gamma} ).
    \end{equation*}
    In the case $d=3$ the index is given by
    \begin{equation*}
      \ind \mathcal{F}_\gamma =
      \begin{cases}
        0 &\text{ if } \gamma-3/2\in (-1, 0)  \\
        -1 &\text{ if } \gamma-3/2\in (0, \epsilon) .
      \end{cases}
    \end{equation*}
    When $d=2$, instead, 
    \begin{equation*}
      \ind \mathcal{F}_\gamma =
      \begin{cases}
        1 &\text{ if } \gamma-1\in (-1, 0)  \\
        -1 &\text{ if } \gamma-1\in (0, \epsilon) .
      \end{cases}
    \end{equation*}

      Let us first consider the case $\gamma-3/2\in(-1,0)$ and $d=3$. The operator $\mathcal{F}_\gamma$ is coercive on $H^1(\Omega) = \mathcal{J}^1_1(\Omega)=\mathcal{K}^1_1(\Omega)$. It is then an isomorphism between the spaces $\mathcal{K}^2_1(\Omega)$ and $\mathcal{K}^{0}_{-1}(\Omega)\times\mathcal{K}^{3/2-m}_{1/2-m}(\partial\Omega)$. Therefore, $\mathcal{F}_\gamma$ is an isomorphism between the spaces \eqref{eq:spaces} for all $-1 <\gamma - 3/2 < 0$, see \cite[Corollary 6.3.3]{Kozlov1997}.

      In the case where $\gamma =1$ and $d=2$, the uniqueness of the solution in $H^1(\Omega)$ implies that the operator is an isomorphims between the spaces \eqref{eq:spaces}.

    Let us now consider the case $\gamma-d/2\in (0,\varepsilon)$ and go back to the generic case $d=2,3$. We introduce $\beta$ such that $\beta-d/2\in (-1,0)$ and consider a solution $u \in \mathcal{K}^s_{\beta}(\Omega)\cap H^1(\Omega)$, $s\geq 2$, to 
    \begin{equation*}
      \begin{aligned}
        L_\beta u &=f  &&\text{ in }\Omega   \\
        B_\beta u &= g &&\text{ on } \partial\Omega
      \end{aligned}
    \end{equation*}
    for $(f,g)\in \mathcal{J}^{s-2}_{\gamma-2}(\Omega)\times \mathcal{J}^{s-m-1/2}_{\gamma-m-1/2}(\partial\Omega)$. The Mellin transform $\mathfrak{L}(\lambda)$ of the principal part of $L$ has a single zero at $\lambda=0$ if $d=3$ and a double zero if $d=2$. We can decompose $u$ as  
  \begin{equation*}
    u = w + u(\mathfrak{c}) 
  \end{equation*}
  where  $w \in \mathcal{K}^{s}_{\gamma}(\Omega)$. This is straightforward for $d=3$; for $d=2$ there could be a term proportional to $\log(r)$, but this term would not belong to $H^1(\Omega)$. Then, $w$ is solution to 
\begin{equation*}
      \begin{aligned}
        L_\gamma w &= f - V u (\mathfrak{c})&&\text{ in }\Omega   \\    
        B_\gamma w & = g - B_\gamma u(\mathfrak{c})&&\text{ on } \partial\Omega 
      \end{aligned}
    \end{equation*}
In this case $\ind \mathcal{F}_\gamma=-1$ but the right hand side in the above equation belongs to the image of $\mathcal{F}_\gamma$, by definition. Furthermore, $f-Vu(\fc)\in\mathcal{K}^{s-2}_{\gamma-2}(\Omega)$ and $g-B_\gamma u(\fc)\in \mathcal{K}^{s-1/2-m}_{\gamma-1/2-m}(\partial\Omega)$. Therefore,
    \begin{equation*}
      \| w\|_{\mathcal{K}^s_\gamma(\Omega)} \leq C\left( \|f\|_{\mathcal{K}^{s-2}_{\gamma-2} (\Omega)} +\| g- Bu(\fc)\|_{\mathcal{K}^{s-1/2-m}_{\gamma-1/2-m}(\partial\Omega)}+ |u(\mathfrak{c})|\right)
    \end{equation*}
    We now conclude as in \cite[Theorem 7.1.1]{Kozlov1997}: since for any $\delta>0$, there exists a $C_\delta$ such that
    \begin{equation*}
      | u(\mathfrak{c})| \leq \delta \| u \|_{\mathcal{J}^2_\gamma(\Omega)} + C_\delta \| u \|_{\mathcal{J}^1_{\gamma-1}(\Omega)} ,
    \end{equation*}
    we can write, for $s\geq 2$,
    \begin{align*}
      \|u \|_{\mathcal{J}^s_\gamma(\Omega)} &\leq C \left( \|w\|_{\mathcal{K}^{s}_\gamma(\Omega) }+ |u(\mathfrak{c})|\right)\\
                                            & \leq C\left(\delta \| u \|_{\mathcal{J}^2_\gamma(\Omega)} + \|f\|_{\mathcal{K}^{s-2}_{\gamma-2} (\Omega)}+ C_\delta \| u \|_{\mathcal{J}^1_{\gamma-1}(\Omega)} + \| g \|_{\mathcal{J}^{s-m-1/2}_{\gamma-m-1/2}(\partial\Omega)}\right) .
    \end{align*}
    Since $\gamma-1 \leq d/2$, by the arguments of the first part of the proof
    \begin{align*}
      \| u \|_{\mathcal{J}^1_{\gamma-1}(\Omega)} \leq  \| u \|_{\mathcal{J}^2_{\gamma-1}(\Omega)}
      &\leq C\left( \|f\|_{\mathcal{J}^0_{\gamma-3}(\Omega)} + \|g\|_{\mathcal{J}^{3/2-m}_{\gamma-3/2-m}(\Omega)} \right)\\
      & \leq C\left( \|f\|_{\mathcal{K}^{s-2}_{\gamma-2} (\Omega)}+\| g \|_{\mathcal{J}^{s-1/2-m}_{\gamma-1/2-m}(\partial\Omega)}\right)
    \end{align*}
    for all $s\geq 2$. The choice of a sufficiently small $\delta$ then concludes the proof.
                                                                    \end{proof}
In the following lemma we extend the estimates for corner domains developed in
\cite[Theorem 3.7]{Costabel2012} to the case of an operator with singular
potential in three dimensions. The proof follows directly from the one in the
cited reference and is therefore omitted.
Let $\gamma\in I_d$ and let $Lg \in \mathcal{K}^{\infty}_{\gamma-2}(\Omega)$. We consider a dyadic decomposition of $\Omega$ given by
  \begin{equation*}
    \Omega_n  = 
     \left\{ x\in \Omega: 2^{-n-1}  < \|x\|_{\ell^\infty} < 2^{-n} \right\}, \quad n\geq 1
  \end{equation*}
  and denote $\Omega_n'$ as the interior of $\overline{\Omega}_{n-1} \cup \overline{\Omega}_n \cup \overline{\Omega}_{n+1}$.

\begin{lemma}
  \label{lemma:elliptic-regularity} 
    For any $n\ge 2$ and $s\ge 2$, the estimate
  \begin{equation}
    \label{eq:elliptic-regularity}
    \wseminorm{g}{s}{2}{\gamma}{\Omega_n} \leq C^s s! \left( \sum_{j=1}^{s-2} \frac{1}{j!} \wseminorm{Lg}{j}{2}{\gamma-2}{\Omega_n'} +\wnorm{{g}}{1}{2}{\gamma}{\gamma-\epsilon/2}{\Omega_n'}\right)
  \end{equation}
  holds, with $C$ independent of $s$ and $n$.
\end{lemma}
We now prove an embedding result that bounds $L^\infty(\Omega)$ norms in weighted spaces with norms of higher derivatives for $p=2$. This is simply the weighted version of the classical embedding of $H^s(\Omega)$ into $L^\infty(\Omega)$ for $s>d/2$, and the proof follows almost directly via dyadic decomposition.
\begin{lemma}
  \label{lemma:embedding-weighted}
  For any $\gamma - d/2\notin \mathbb{N}$ and for any $t>s+d/2$ there exists
  $C>0$ such that for any $u\in\mathcal{J}^{t}_{\gamma}(\Omega)$,
  \begin{equation}
    \label{eq:embedding-weighted}
    \| u \|_{\mathcal{J}^{s,\infty}_{\gamma-d/2}(\Omega)} \leq C \| u \|_{\mathcal{J}^{t}_{\gamma}(\Omega)}.
  \end{equation}
  \end{lemma}
\begin{proof}
  To prove \eqref{eq:embedding-weighted} we use the fact that $\mathcal{J}^t_\gamma (\Omega)= \mathcal{K}^t_\gamma (\Omega) \oplus \mathbb{Q}_{\lfloor \gamma-d/2\rfloor}(\Omega)$ and decompose $u = v + w$  such that
  \begin{equation*}
    v\in \mathcal{K}^{s,2}_{\gamma}(\Omega) \quad \text{ and }\quad  w \in \mathbb{Q}_{\lfloor\gamma-d/2\rfloor}(\Omega).
  \end{equation*}
  Furthermore, we have
  \begin{equation*}
    \| u \|_{\mathcal{J}^t_\gamma(\Omega)} \simeq \| v \|_{\mathcal{K}^t_\gamma(\Omega)} + \| w\|_{\mathbb{Q}_{\lfloor \gamma-d/2\rfloor}(\Omega)}
  \end{equation*}
for any chosen norm $ \| \cdot \|_{\mathbb{Q}_{\lfloor \gamma-d/2\rfloor}(\Omega)}$, thanks to the equivalency of norms in finite dimensional spaces, see \cite{Costabel2010a} and \cite[Theorem 7.1.1]{Kozlov1997}. Then, by the triangle inequality and the definition of the norms in the weighted spaces,
\begin{align*}
    \| u \|_{\mathcal{J}^{s,\infty}_{\gamma-d/2}(\Omega)}  
&\leq  \| v \|_{\mathcal{J}^{s,\infty}_{\gamma-d/2}(\Omega)}  + \|w\|_{\mathcal{J}^{s,\infty}_{\gamma-d/2}(\Omega)}\\
&\leq C \| v \|_{\mathcal{K}^{s,\infty}_{\gamma-d/2}(\Omega)}  + \|w\|_{\mathcal{J}^{s,\infty}_{\gamma-d/2}(\Omega)},
\end{align*}
and we consider separately the two terms at the right hand side.
Consider the annuli
  \begin{equation*}
    \Gamma_j = \left\{ x\in \Omega : 2^{-j} < \|x\|_{\ell^\infty} < 2^{-j+1} \right\},\,j \in \mathbb{N}
  \end{equation*}
  and let $\widehat{\Gamma}=\Gamma_0$. Then, scaling and using a Sobolev inequality, 
  \begin{align*}
    \| v \|_{\mathcal{K}^{s,\infty}_{\gamma-d/2} (\Gamma_j)}  &\simeq 2^{j(\gamma-d/2)} \| \hat{v}\|_{W^{s,\infty}(\widehat{\Gamma})}\\
                                                              &\lesssim 2^{j(\gamma-d/2)} \| \hat{v}\|_{H^{t}(\widehat{\Gamma})}\\
                                                              & \lesssim 2^{j(\gamma-d/2)} \| \hat{v} \|_{\mathcal{K}^{t}_{\gamma}(\widehat{\Gamma})} \\
                                                              & \simeq \|v\|_{\mathcal{K}^{t}_{\gamma}(\Gamma_j)}\\
                                                              & \lesssim \| u \|_{\mathcal{J}^{t}_\gamma (\Omega)},
  \end{align*}
  where the quantities with a hat are rescaled on $\widehat{\Gamma}$. Therefore 
  \begin{equation*}
    \| v \|_{\mathcal{K}^{s,\infty}_{\gamma-d/2} (\Omega)}= \sup_j \|v\|_{\mathcal{K}^{s,\infty}_{\gamma-d/2} (\Gamma_j)} \leq C   \| u \|_{\mathcal{J}^{t}_\gamma (\Omega)}   .
  \end{equation*}
  Since $w$ lies in the finite dimensional space of polynomials of degree
  $\lfloor \gamma -d/2 \rfloor$, we can conclude with
  \eqref{eq:embedding-weighted}, where the constant $C$ can depend on the domain
  $\Omega$, on the dimension $d$ and on $\gamma$, but does not depend on $s$ and
  $t$.
\end{proof}
The weighted analytic estimates then follow for $p=\infty$. Lemma
\ref{lemma:embedding-weighted} directly implies the following statement.
\begin{corollary}
  Let $\gamma-d/2\notin \mathbb{N}$. If $u\in \mathcal{J}^{\varpi,2}_{\gamma}(\Omega)$, then $u\in \mathcal{J}^{\varpi,\infty}_{\gamma-d/2}(\Omega)$.
\end{corollary}

It is now evident that, using Lemmas \ref{lemma:solution-regularity} and \ref{lemma:elliptic-regularity}, we can prove that when the right hand side and the potential of \eqref{eq:mainproblem} obey analytic growth estimates on the weighted norms of the derivatives, the solution $u$ is in the same regularity class. This is the content of the following corollary.
\begin{corollary}
  \label{cor:analytic}
  If $u$ is solution to \eqref{eq:mainproblem} with $V:\mathbb{R}^d\to\mathbb{R}_+$ such that
$V\in\mathcal{K}^{\varpi,\infty}_{\epsilon-2}(\Omega)$ for some $0<\epsilon\leq 1$,
then $u\in\mathcal{J}^{\varpi,2}_{\gamma}(\Omega)$ for any $\gamma < d/2+\epsilon$.
\end{corollary}
\begin{proof}
 By the initial regularity of $u\in H^1(\Omega)$ and Lemma
 \ref{lemma:solution-regularity} we have that
 $u\in\mathcal{J}^2_{\gamma}(\Omega)$ for $\gamma\in I_d$. We can decompose
 $u=(u-u(\fc)) + u(\fc)$ and apply \eqref{eq:elliptic-regularity} to $g =
 u-u(\fc)$. Then, 
 $V\in\mathcal{K}^{\varpi,\infty}_{-2+\epsilon}(\Omega)$, and  $|u(\fc)| \leq C$
by Lemma \ref{lemma:embedding-weighted}, hence $(L-\lambda)g = \lambda u(\fc)  -
Vu(\fc) \in\mathcal{K}^{\varpi}_{\gamma-2}(\Omega)$. In addition $\| u - u(\fc)\|_{\mathcal{K}^1_\gamma(\Omega)}\lesssim \| u \|_{\mathcal{J}^{1}_\gamma(\Omega)} + |u(\fc)|$. Summing the left and right hand sides of \eqref{eq:elliptic-regularity} over all $\Omega_k$ gives the existence of $C,A\in \mathbb{R}^+$ such that
  \begin{equation*}
    | u |_{\mathcal{K}^s_\gamma(\Omega)} \leq C A^s s!,
  \end{equation*}
  for all $s\geq 2$ and $\gamma\in I_d$, thus $u\in\mathcal{J}^{\varpi}_\gamma(\Omega)$.
\end{proof}

\section{Numerical approximation}
\label{sec:numerical-appx}
In this section, we consider the approximation of the linear elliptic eigenvalue
problem \eqref{eq:mainproblem} obtained through an isotropically graded discontinuous Galerkin $hp{}$
method.

The contents of
the section are largely based on \cite{Antonietti2006}, where the convergence of
the discontinuous Galerkin method is proven for linear elliptic eigenvalue problems.
The result obtained in that paper is an extension to discontinuous Galerkin
methods of the theory developed almost three decades earlier, see
\cite{Descloux1978a, Descloux1978b}. A thorough presentation of the
approximation of eigenvalue problems is also given in~\cite[Chapter II]{Handbook2}.

The only differences
with the analysis in \cite{Antonietti2006} are due to the presence of a potential and to the specificity of
approximation in isotropically refined $\hp{}$ finite element spaces.

In the following, we introduce the analysis developed in the aforementioned papers,
specializing it to problems with singular points and interior penalty
discontinuous Galerkin methods.

We also introduce the interior penalty discontinuous Galerkin methods that
will be taken into consideration. We will be dealing with a symmetric operator
and a coercive linear form, thus the spectrum  is composed of real isolated
eigenvalues of ascent one. The analysis can still be partially extended
to non-symmetric problems, but it has to be taken into account that the operators
are not self-adjoint. We conclude the section by introducing the ``solution
operators'' $T$, for the continuous problem, and $\Td$, for the discrete approximation.
$T$ and $\Td$ are continuous and invertible operators, with the same eigenspaces
as the ones of the original problems and with reciprocal eigenvalues. The
analysis will center around those operators, and the final results can easily be
applied back to the original problems. Finally, we will need a way to measure a
``distance'' between eigenspaces: this is the role of $\delta(\cdot, \cdot)$ and
$\hd(\cdot, \cdot)$ defined
in \eqref{eq:delta-def}.

The interest of the analysis of the approximation of an eigenvalue problem lies
not only in the convergence of the numerical eigenpairs to the exact ones, but
also in the non pollution and completeness of eigenfunctions and eigenvalues. 
Basically, a good approximation of an eigenproblem should not introduce any
spurious numerical eigenvalue or eigenvector (non-pollution) and should
approximate all eigenpairs (completeness). In Theorem \ref{th:non-poll-spectrum}
we show that the spectrum is not polluted, while in Theorem \ref{th:deltahat}
the completeness of the approximation is shown (more precisely, Theorem
\ref{th:deltahat} gives both completeness and convergence for finite dimensional
eigenspaces, while simple completeness is a consequence of Lemma \ref{lemma:dspace2}).
Note that, in practice, some techniques may still introduce spurious eigenvalues
in the approximation: consider for example the ``strong'' imposition of boundary
conditions in a numerical code, where the matrix resulting from the
approximation of the operator is modified in order to set the degrees of freedom
at the boundary, see, e.g., the documentation of \cite{dealII85}.
This is out of the scope of the present analysis; furthermore,
the spurious eigenvalues can often be easily identified.

Finally, in Section \ref{sec:lin-convergence}, the focus is on the rate of
convergence of the numerical eigenpairs. We consider finite dimensional exact
eigenspaces and we introduce a projector from the exact to the numerical
eigenspace, thus obtaining a somewhat algebraic problem, at least in the
relationship between the eigenvalues and the (projected) operators $\Th$ and
$\Thd$ (the latter can be seen as tensors in the finite dimensional eigenspace).
We obtain the expected quasi optimal estimates on the difference between exact and numerical
eigenfunctions. The eigenvalue error, additionally, can be shown to converge
with a higher rate of convergence --- quadratically with respect to the
eigenfunctions --- if the method is adjoint consistent (symmetric, in our case).

We now introduce the discontinuous Galerkin interior penalty method.

\subsection{Interior penalty method}
Let $\mathcal{T}$ be a mesh isotropically and geometrically graded around the
points in $\mathfrak{C}$. We assume that the mesh is shape- and contact-regular
and we indicate by $\Omega_j$, $j=1, \dots, \ell$, the set of elements and edges
at the same level of refinement.We introduce on this mesh the \hp{} space with refinement ratio $\sigma$ and linear polynomial slope $\slope$, i.e., for an element $K\in \mathcal{T}$ such that $K\in\Omega_j$,
\begin{equation*}
  h_K \simeq h_j = \sigma^{j} \text{ and } p_K \simeq p_j = p_0+\slope(\ell-j),
\end{equation*}
where $h_K$ is the diameter of the element $K$ and $p_K$ is the polynomial order
whose role will be specified in \eqref{eq:lin-discrete-space}. We suppose that for any $K\in\mathcal{T}$ there exists an affine transformation $\Phi:K\to\hK$ to the $d$-dimensional cube $\hK$ such that $\Phi(K) = \hK$, and introduce the discrete space 
\begin{equation}
  \label{eq:lin-discrete-space}
  \de{X} = \left\{ \de{v} \in L^2(\Omega):(v_{|_K}\circ\Phi^{-1})\in \mathbb{Q}_{p_K}(\hK)\; \forall K\in\mathcal{T}\right\},
\end{equation}
where $\mathbb{Q}_p$ is the space of polynomials of maximal degree $p$ in any variable.
Let then $\mathcal{E}$ be the set of the edges (for $d=2$) or faces ($d=3$) of the elements in $\mathcal{T}$ and
\begin{align*}
  \dhe &= \min_{K\in \mathcal{T}: e\cap\partial K \neq \myempty}h_K \\
  \dpe &= \max_{K\in \mathcal{T}: e\cap\partial K \neq \myempty}p_K .
\end{align*}
Note that edges and faces are open $d-1$ dimensional sets.
On an edge/face between two elements $K_\sharp$ and $K_\flat$, i.e., on $e\subset\partial{K}_\sharp\cap \partial{K}_\flat$, the average $\avg{\cdot}$ and jump $\jump{\cdot}$ operators for a function $w\in X(\delta)$  are defined by
\begin{equation*}
  \avg{w} = \frac{1}{2}\left( w_{|_{K_\sharp}}+w_{|_{K_\flat}}\right), \qquad \jump{w} =   w_{|_{K_\sharp}}\mathbf{n}_\sharp+w_{|_{K_\flat}}\mathbf{n}_\flat,
\end{equation*}
where $\mathbf{n}_\sharp$ (resp. $\mathbf{n}_\flat$) is the outward normal to the element $K_\sharp$ (resp. $K_\flat$). In the following, for an $S\subset \Omega$, we denote by $(\cdot, \cdot)_S$ the $L^2(S)$ scalar product and by $\| \cdot \|_S$ the $L^2(S)$ norm. 

We indicate by $\ad(\cdot, \cdot):\Xd\times\Xd\to\mathbb{R}$ the interior penalty bilinear form, given by
\begin{equation}
  \begin{aligned}
  \label{eq:bilin-discont}
  a_\delta (u_\delta,v_\delta) = (\nabla u_\delta, \nabla v_\delta)_{\mathcal{T}} &- (\avg{\nabla u_\delta},\jump{v_\delta})_{\mathcal{E}_I}- \theta(\avg{\nabla v_\delta},\jump{u_\delta})_{\mathcal{E}} \\ & + \sum_{e \in \mathcal{E}}\alpha_e \frac{\dpe^2}{\dhe} (\jump{u_\delta}, \jump{v_\delta})_{e} + \int_\Omega V u_\delta v_\delta.
\end{aligned} 
\end{equation}
Here, $\mathcal{E}_I$ is the set of internal edges such that for all
$e\in\mathcal{E}_I$, $e\cap \partial \Omega=\emptyset$, and we have written
\begin{equation*}
  (\cdot, \cdot)_{\mathcal{T}} = \sum_{K\in \mathcal{T}}(\cdot, \cdot)_K\qquad
  (\cdot, \cdot)_{\mathcal{E}} = \sum_{e\in \mathcal{E}}(\cdot, \cdot)_e.
\end{equation*}
The discrete eigenvalue problem then reads: find $(\lambda_\delta, u_\delta)\in
\mathbb{C}\times \Xd$
\begin{equation}
    \label{eq:discrete-prob}
    \ad(\ud, \vd)  = \ld (\ud,\vd) \text{ for all }\vd\in\Xd.
  \end{equation}
Choosing $\theta = 1$ in \eqref{eq:bilin-discont}
gives the symmetric interior penalty (SIP) method, while $\theta=-1$ gives the
non-symmetric interior penalty (NIP) method, and $\theta=0$ gives the incomplete
interior penalty method (IIP). We remark that the choice $\theta=1$ is the only
one that assures the symmetry of the method; the SIP method is \emph{adjoint
  consistent}. 

We write $X=H^1(\Omega)$, $X(\delta) = X+\Xd$ and introduce the mesh dependent norms
\begin{align*}
  \dgnorm{v}^2
 & = \sum_{K\in \mathcal{T}} \|v \|^2_{H^1(K)} + \sum_{e\in\mathcal{E}}\dpe^2\dhe^{-1}\| \jump{v}\|^2_{L^2({e})}
\intertext{and}
    \fullnorm{v}^2
  &= \dgnorm{v}^2 + \sum_{e\in\mathcal{E}} \dhe \dpe^{-2} \| \nabla v \|^2_{L^2(e)}.
\end{align*}
Note that $\dgnorm{\cdot}$ is defined on $\XXd$, while $\fullnorm{\cdot}$ is
defined only on the broken space
\[\XXd \cap H^{d/2}(\mathcal{T})  = \left\{ v\in \XXd : v\in H^{d/2}(K) \text{
      for all }K \in \mathcal{T}\right\},\]
due to the presence of the boundary gradient term.
We introduce the continuous solution operator
\begin{equation}
  \label{eq:Tdef}
  T : L^2(\Omega) \to X
\end{equation}
such that
\begin{equation*}
 a(Tu, v) = (u,v),  \text{for all } v \in X
\end{equation*}
and its discrete counterpart, given by
\begin{equation}
  \label{eq:Tddef}
  \Td : L^2(\Omega) \to \Xd
\end{equation}
such that
\begin{equation*}
 \ad(\Td u, v) = (u,v),  \text{for all } \vd \in \Xd.
\end{equation*}
The analysis of the relation between the spectra associated to the operator $L$
in \eqref{eq:mainproblem} and to the discrete bilinear form $\ad$ can be
reconducted to the analysis of the spectra of $T$ and $\Td$. Since the bilinear form associated to $L$ is continuous, coercive on $X$, and symmetric,
\begin{enumerate}[(i)]
\item all the eigenvalues are real and strictly positive,
\item the set of the eigenvalues of $L$ is a countably infinite sequence diverging to $\infty$,
\item all the eigenspaces are finite dimensional,
\item eigenfunctions associated with different eigenvalues are L2-orthogonal,
\item the eigenfunctions are complete $L^2(\Omega)$ and in $X$.
\end{enumerate}
    
 In the following, the
spectrum of $T$ will be denoted by $\sigma(T)$ and its resolvent set by
$\rho(T)$.
Similarly, $\sigma(\Td)$ and $\rho(\Td)$ will be respectively the
spectrum and resolvent set of $\Td$. Let then 
\begin{equation*}
  R_z(T) = (z-T)^{-1}
\end{equation*}
be the resolvent operator associated with $T$, and
\begin{equation*}
  R_z(\Td) = (z-\Td)^{-1}
\end{equation*}
be the resolvent operator associated with $\Td$, both defined for
$z\in\mathbb{C}$. Finally, we introduce a measure of the gap between subspaces
of $\XXd$: let $Y$ and $Z$ be close subspaces of $\XXd$; then for an $x\in X$
\begin{equation}
  \label{eq:delta-def}
  \begin{aligned}
  &\delta(x, Y)  = \inf_{y\in Y} \dgnorm{x - y}, \qquad \delta(Y,Z) = \sup_{y\in Y:\dgnorm{y}=1} \delta(y, Z)\\
  &\hd(Y,Z) = \max(\delta(Y,Z), \delta(Z,Y))
  \end{aligned}
\end{equation}

\subsection{Non pollution and completeness of the discrete spectrum and eigenspaces}
\label{sec:nonpoll-complet}
\subsubsection{Non pollution of the spectrum}
In this section we detail the technique used in \cite{Antonietti2006} to prove
the non-pollution of the discrete spectrum.
Note that, thus far, $R_z(\Td)$ has only been defined formally. We will now show
its existence and continuity, together with the existence and continuity of its inverse.
This will imply the non pollution of the discrete spectrum and guarantee that,
for a sufficient number of degrees of freedom, the discrete spectrum lies in the
vicinity of the continuous one.

We start by introducing a lemma, whose proof
we postpone to the end of the section.
\begin{lemma}
  \label{lemma:nonpoll1}
  Let $z\in \rho(T)$ such that $z\neq 0$ and $u\in X(\delta)$. Then, there
  exists $C>0$ such that
  \begin{equation*}
    \dgnorm{(z-T) u} \geq C \dgnorm{u}
  \end{equation*}
  where $C$ depends on $L$, on $\Omega$, and on $|z|$.
\end{lemma}
By the triangle inequality, then,
\begin{equation}
  \label{eq:nonpoll1}
  \dgnorm{ (z-\Td) u} \geq \dgnorm{ (z-T)u} - \dgnorm{(T-\Td)u}.
\end{equation}
Now, the second term at the right hand side is the classical error of the
method; by the coercivity and continuity of the discrete bilinear form, Lemma
\ref{lemma:solution-regularity} and the approximation properties of the \hp{} space, we have
that
\begin{equation*}
  \dgnorm{(T-\Td)u} \to 0 \text{ as }N\to\infty 
\end{equation*}
where $N$ is the dimension of $\Xd$.
Using Lemma \ref{lemma:nonpoll1} and the above estimate in \eqref{eq:nonpoll1},
we obtain that, for a sufficient number of degrees of freedom,
\begin{equation}
  \label{eq:nonpoll2}
  \dgnorm{ (z-\Td) u} \geq C \dgnorm{u}
\end{equation}
for $0\neq z \in \rho(T)$. For a fixed $z$ and for a sufficient number of
degrees of freedom (depending on $z$), thus, $z-\Td$ is invertible and
$R_z(\Td)$ is well defined. Furthermore, Lemma \ref{lemma:nonpoll1} implies that
$R_z(T)$ is well defined and bounded as an operator on the spaces $\XXd\to\XXd$.
We have therefore shown that $R_z(\Td)$ is bounded as a linear operator from
$\XXd$ to $\XXd$, and that the spectrum is not polluted; in the following we summarize this results.
Denoting by $\| \cdot\|_{\mathcal{L}(V,W)}$ the classical operator norm
\begin{equation*}
  \| F \|_{\mathcal{L}(V,W)} = \sup_{v\in V: \|v\|_{V}=1} \|Fv\|_W,
\end{equation*}
from \eqref{eq:nonpoll2} we conclude that
\begin{lemma}
  \label{lemma:nonpoll2}
  Let $A\subset \rho(T)$ be a closed set. Then, for all $z\in A$, there exists a
  constant $C$ such that
  \begin{equation*}
    \| R_z(\Td)\|_{\mathcal{L}(\XXd, \XXd)} \leq C.
  \end{equation*}
\end{lemma}
The non-pollution of the spectrum follows directly, taking the complementary of
the set $A$ above.
\begin{theorem}
  \label{th:non-poll-spectrum}
  Let $B\supset\sigma(T)$ be an open set. Then, for a sufficient number of
  degrees of freedom,
  \begin{equation*}
    \sigma(\Td)\subset B.
  \end{equation*}
\end{theorem}
We conclude the section with the proof of Lemma \ref{lemma:nonpoll1}.
\begin{proof}[Proof of Lemma \ref{lemma:nonpoll1}]
  Consider $u\in \XXd$ and $0\neq z \in \rho(T)$. Then, by the triangle inequality,
  \begin{equation}
    \label{eq:lemma-nonpoll1}
    |z|\dgnorm{u} \leq \dgnorm{zTu} + \dgnorm{(z-T)u}.
  \end{equation}
  Let now $v=zTu$. Then, by the definition of $T$, $zu = Lv$ and
  \begin{equation*}
    L v - \frac{1}{z} v = (z - T)u 
  \end{equation*}
  with the associated boundary conditions. Since $z\in\rho(T)$, the operator
  $L-1/z$ is invertible, and
  \begin{equation*}
    \dgnorm{zTu} = \| v\|_X \leq C \| (z-T) u\|_{L^2(\Omega)}.
  \end{equation*}
  The constant $C$ clearly depends on $z$, on the operator $L$, and on $\Omega$.
  Inserting the above inequality into \eqref{eq:lemma-nonpoll1} one obtains the thesis.
\end{proof}
\subsubsection{Eigenspaces and completeness of the spectrum}
Consider a smooth closed curve $\Gamma\subset \rho(T)$. We introduce the spectral
projectors
\begin{equation}
  \label{eq:projectors}
  E = \frac{1}{2\pi i} \int_\Gamma R_z(T)dz \quad \text{and}\quad \Ed = \frac{1}{2\pi i} \int_\Gamma R_z(\Td)dz
\end{equation}
Clearly, both projectors depend on $\Gamma$, we omit that in our notation as is
customary: suppose that $\Gamma$ is fixed and that it encloses a single
eigenvalue of $T$.
The discrete projector $\Ed$ is, once again, well defined provided that the
space $\Xd$ contains a sufficient number of degrees of freedom. Suppose that
$\Gamma$ contains an eigenvalue of $T$; then, $E$ is the projector on the
eigenspace associated to the eigenvalue. The same holds for the discrete version.

We now wish to prove the convergence of the discrete projector to the continuous
one, in the operator norm. We start by noting that
\begin{equation*}
  (z-T)^{-1} - (z-\Td)^{-1} = (z-\Td)^{-1}(T-\Td)(z-T)^{-1},
\end{equation*}
therefore,
\begin{align*}
  \|R_z(T) - R_z(\Td)\|_{\mathcal{L}(L^2(\Omega), \XXd)}
  & = \|R_z(\Td)(T-\Td) R_z(T)\|_{\mathcal{L}(L^2(\Omega), \XXd)}\\
& \leq \|R_z(\Td)\|_{\XXd, \XXd)}\|(T-\Td)\|_{\mathcal{L}(L^2(\Omega), \XXd)}\| R_z(T)\|_{L^2(\Omega), L^2(\Omega)}.
\end{align*}
Due to the boundedness of the continuous, see \cite{Antonietti2006}, and discrete resolvent operators, see
Lemma \ref{lemma:nonpoll2}, we conclude that
\begin{equation}
  \label{eq:E-T}
  \|E-\Ed\|_{\mathcal{L}(L^2(\Omega), \XXd)} \leq C \|(T-\Td)\|_{\mathcal{L}(L^2(\Omega), \XXd)}.
\end{equation}
\begin{lemma}
  \label{lemma:proj-conv}
 Given the definition of $E$ and $\Ed$ in \eqref{eq:projectors}, if $\Xd$ has a
 sufficient number of degrees of freedom, there holds
 \begin{equation*}
   \|E-\Ed\|_{\mathcal{L}(L^2(\Omega), \XXd)} \to 0.
 \end{equation*}
  \end{lemma}
Consider now the definitions given in \eqref{eq:delta-def}. The convergence of
the projectors allows for the proof of the convergence to zero of some
``distances'' between eigenspaces. The first almost direct result is in the
following lemma.
\begin{lemma}
  \label{lemma:dspace1}
  Let $\delta(\cdot, \cdot)$ be defined as in \eqref{eq:delta-def}. Then,
  \begin{equation*}
    \delta(\Ed(\Xd), E(X)) \to 0
  \end{equation*}
\end{lemma}
\begin{proof}
  For any $\xd\in \Ed(\Xd)$, $\Ed(\xd) = \xd$. We remark that, due to the
  regularity result given in Lemma \ref{lemma:solution-regularity}, $E(L^2(\Omega)) = E(X)$.
  Thus, for any $\xd\in \Ed(\Xd)$ such that $\dgnorm{\xd} =1$,
  \begin{align*}
    \inf_{x\in E(X)}\dgnorm{\xd-x}
    &= \inf_{x\in E(L^2(\Omega))}\dgnorm{\xd-x}\\
    & = \inf_{y\in L^2(\Omega)} \dgnorm{\Ed \xd - E y}\\
    & \leq \| \Ed- E\|_{\mathcal{L}(L^2(\Omega), \XXd)}.
  \end{align*}
 Taking the supremum over all $\xd\in\Ed(\Xd)$ one obtains the thesis.
                  \end{proof}
This is a proof of the non pollution of the eigenspaces: we have indeed shown
that all numerical eigenfunction converge to an exact one.
We continue by
showing the completeness of the eigenspaces. This involves proving that any
exact eigenfunction is approximated by a numerical one.
\begin{lemma}
  \label{lemma:dspace2}
  For any $x\in E(X)$,
  \begin{equation*}
    \delta(x, \Ed(\Xd)) \to 0
  \end{equation*}
\end{lemma}
\begin{proof}
  Let $x\in E(X)$ and $\xd\in\Xd$. Then,
  \begin{equation*}
    \dgnorm{\Ed x_\delta - x} \leq \|E\|_{\mathcal{L}(\XXd, \XXd)}\fullnorm{x_\delta -x}+\|E - \Ed\|_{\mathcal{L}(\XXd, \XXd)}\dgnorm{\xd}.
  \end{equation*}
  Taking $\xd$ as the projection of $x$ in $\Xd$ and thanks to the convergence
  of $\Ed$ towards $E$, we obtain the thesis.
\end{proof}
We now restrict our focus to finite dimensional eigenspaces. Let then $n =
\dim(E(X))$ and $n_\delta = \dim(\Ed(\Xd))$: if $n=\infty$, then $n_\delta \to
\infty$; we consider the case where $n$ is finite.
If $n$ is finite, the above lemma implies that
\begin{equation*}
  \delta(E(X), \Ed(\Xd)) \to 0.
\end{equation*}

\begin{remark}
  \label{remark:reg-resolvent}
  The eigenspace $E(X)$ is invariant for $T$, hence if $x\in E(X)$, then
  $R_z(T)x\in E(X)$.
\end{remark}

Consider then an $x\in E(X)$: we have
\begin{equation}
  \label{eq:exprate1}
  \inf_{x_\delta\in \Xd}\dgnorm{\Ed x_\delta - x} \leq \|\Ed \|_{\mathcal{L}( \XXd, \XXd)} \inf_{x_\delta\in \Xd}\fullnorm{ x_\delta -x } + \dgnorm{ \left( E- E_\delta \right)x} 
\end{equation}
 Due to the approximation properties
 of $\Xd$ there exist $C,b>0$ such that
 \begin{equation}
  \label{eq:exprate2}
   \inf_{x_\delta\in \Xd}\fullnorm{x - x_\delta} \leq C e^{-b N^{1/(d+1)}},
 \end{equation}
 with $N=\dim(\Xd)$. In addition,
 \begin{align*}
\sup_{\substack{x\in E(X)\\\|x\|=1}}\fullnorm{\left( R_z(T) - R_z(\Td) \right)x}
  & =\sup_{\substack{x\in E(X)\\\|x\| = 1}} \fullnorm{ \left( R_z(\Td)(T-\Td) R_z(T)\right)x}\\
  & \leq C \|R_z(\Td)\|_{\mathcal{L}(\XXd, \XXd)} \\
   &\qquad \times \|(T-\Td)_{|_{E(X)}}\|_{\mathcal{L}(L^2(\Omega), \XXd))} \|R_z(T)\|_{\mathcal{L}(L^2(\Omega), L^2(\Omega))},
  \end{align*}
As above, the boundedness of the continuous, see \cite{Antonietti2006}, and discrete resolvent operators, see
Lemma \ref{lemma:nonpoll2}, imply that
\begin{equation}
  \label{eq:projector-bound}
  \sup_{\substack{x \in E(X)\\\|x\| = 1}}\fullnorm{(E- E_\delta)x} \leq C\sup_{\substack{x \in E(X)\\\|x\| = 1}}\fullnorm{(T-\Td)x}
  \end{equation}
Thanks to Remark \ref{remark:reg-resolvent}, the right hand side of the above equation is the error of the numerical method
for a problem with source term belonging to $\mathcal{J}^\varpi_\gamma(\Omega)$: by Lemma \ref{lemma:solution-regularity},
the approximation properties of the \hp{} space, and the compactness of the
unitary ball in the finite dimensional space $E(X)$,
there exist $C,b>0$ such that
\begin{equation}
  \label{eq:exprate3}
  \sup_{\substack{x\in E(X)\\\|x\| = 1}}\fullnorm{(E- E_\delta)x} \leq C e^{-b N^{1/(d+1)}}.
\end{equation}
 Combining \eqref{eq:exprate1}, \eqref{eq:exprate2} and \eqref{eq:exprate3}, we
 have then the explicit rate
\begin{equation*}
  \delta(E(X), \Ed(\Xd)) \leq C e^{-bN^{1/(d+1)}}.
\end{equation*}
We summarize this in the following statement.
\begin{theorem}
  \label{th:deltahat}
  If $\dim(E(X))< \infty$ and for a sufficient number of degrees of freedom,
  there exist $C,b>0$ such that
  \begin{equation*}
    \delta(E(X), \Ed(\Xd)) \leq C e^{-bN^{1/(d+1)}}.
  \end{equation*}
      \end{theorem}
\subsection{Convergence of the eigenfunctions and eigenvalues}
\label{sec:lin-convergence}
In this section we consider the convergence of the numerical
eigenfunctions and eigenvalues obtained through the $\hp{}$ approximation. As
far as the eigenspaces are concerned, Lemma \ref{lemma:dspace1} proves that they
are not polluted and Lemma \ref{lemma:dspace2} proves that they are complete. 
As a direct consequence of Theorem \ref{th:deltahat}, furthermore, we have that
for any $x\in E(X)$ there exists $\xd\in \Ed(\Xd)$
such that
\begin{equation*}
  \dgnorm{x-\xd}\leq C e^{-bN^{1/(d+1)}}.
\end{equation*}

We now consider the eigenvalues; we will do so in the case of a symmetric
numerical scheme.
\subsubsection{Convergence of the eigenvalues}
We are mainly interested in the analysis of the convergence of the eigenvalues for the symmetric interior penalty method, obtained
by choosing $\theta=1$ in \eqref{eq:bilin-discont}. For the sake of generality, the first part of the section will, nonetheless, hold
for non-symmetric methods and for a non symmetric operator, but we will indicate when the hypothesis of symmetry
of the numerical method will become necessary. The final result
obtained for the SIP method will be stronger than what can
be obtained in the case of non-symmetric methods, since
they lack the property of adjoint consistency.

We start by considering the operator $\Ld = \Ed_{|_{E(X)}}: E(X)\to \Ed(\Xd)$.
For a sufficient number of degrees of freedom, the operator is invertible.
For any $x\in E(X)$,
\begin{equation*}
  \dgnorm{x} \leq \dgnorm{(E-\Ed)x} + \dgnorm{\Ed x}
\end{equation*}
and the convergence of $E-\Ed$ in the operator norm implies that for a
sufficient number of degrees of freedom, $\Ld^{-1}$ is bounded.
Let us then introduce the operators
\begin{equation}
  \Th = T_{|_{E(X)}} \quad \text{and} \quad\Thd = \Ld^{-1}\Td \Ld,
\end{equation}
both defined on the spaces $E(X)\to E(X)$. We consider the case where $\Gamma$,
introduced in \eqref{eq:projectors},
contains a single eigenvalue $\mu$ of $T$, with multiplicity $n$. Theorem
\ref{th:deltahat} then implies (see \cite{Antonietti2006} for the details) that there exist 
$\mdj$, $j=1,\dots,n$ that converge towards $\mu$.
For every $\mdj$ there exists, then, an $x_j\in E(X)$ such that
\begin{equation*}
  \Thd x_j = \mdj x_j.
\end{equation*}
Let now $T'$ and $\Td'$ be the adjoint operators to $T$ and $\Td$, and let $E'$
and $\Ed'$ be the associated spectral projectors.
Furthermore, consider a $y\in E'(X)$ such that $(x,y) = 1$: since for all $x\in E(X)$, $(T-\mu)x =
0$ (since all eigenvalues have ascent one), we have
\begin{align*}
  \mu - \mdj
  &= \langle (\mu - \Thd) x, y \rangle\\
  & =\langle  (T-\Thd)x, y \rangle \\
  & =\langle  (T-\Ld^{-1}\Td \Ed)x, y \rangle 
\end{align*}
Note now that $\Ld^{-1}\Ed_{|_{E(X)}} = I$ and that $\Td$ and $\Ed$ commute on
$E(X)$, thus
\begin{align*}
  \mu - \mdj
  &= \langle (\Ld^{-1}\Ed)(T-\Td) x,y \rangle\\ 
  &= \langle (T-\Td) x,y \rangle + \langle (\Ld^{-1}\Ed-I)(T-\Td)x,y \rangle.
\end{align*}
We remark that $\ker((\Ld^{-1}\Ed-I)_{|_{E(X)}}) =
\ker(\Ed)^\perp$, hence \[\Ld^{-1}\Ed-I : E(X)\to \im(\Ed')^\perp.\]
Using also the fact that $E'y = y$, the second term at the right hand side above can be written as
\begin{equation*}
  \langle (\Ld^{-1}\Ed-I)(T-\Td)x,y \rangle = \langle (\Ld^{-1}\Ed-I)(T-\Td)x,(E'-\Ed')y \rangle.
\end{equation*}
As already shown $\Ld^{-1}$ is bounded for a sufficient number of degrees of
freedom, and so is $\Ed$. Let us
now choose
$\|x\| = \|y\| = 1$: we have
\begin{equation*}
  \left|  \langle (\Ld^{-1}\Ed-I)(T-\Td)x,y \rangle \right|\leq  C \sup_{\substack{x\in E(X)\\\|x\|=1}}\fullnorm{(T-\Td)x}\sup_{\substack{y\in E'(X)\\\|y\|=1}}\fullnorm{(T'-\Td')y}
\end{equation*}
where
we have used \eqref{eq:projector-bound} for the adjoint spectral projectors.
We introduce two orthonormal bases $\{\phi_i\}_i$ and $\{\phi'_j\}_j$
for $E(X)$ and $E'(X)$ respectively. Since the spaces are finite dimensional,
i.e., $n=\dim(E(X))=\dim(E'(X))< \infty$, there exists a
constant $C>0$
such that 
\begin{align*}
\langle \left( T - \Td \right)x, y \rangle
  & \leq \sup_{\|x\|=\|y\| = 1} \left| \langle \left( T - \Td \right)x, y \rangle \right|\\
 & \leq C \sum_{i,j=1}^n  \left| \langle \left( T - \Td \right)\phi_i, \phi'_j \rangle \right|,
\end{align*}
where $C$ depends on $n$.
We conclude that
\begin{equation}
  \label{eq:eigenvalue1}
  |\mu - \mdj| \leq C \left(  \sum_{i,j=1}^n \langle
  (T-\Td)\phi_i, \phi'_j \rangle + \sup_{\substack{x\in E(X)\\\|x\|=1}}\fullnorm{(T-\Td)x}\sup_{\substack{y\in E'(X)\\\|y\|=1}}\fullnorm{(T'-\Td')y}\right),
\end{equation}

\begin{remark}
  The above estimate \eqref{eq:eigenvalue1} holds since we have considered a
  case where all eigenvalues have ascent one. If this were not the case, one
  would find that 
  \begin{equation*}
  |\mu - \frac{1}{n}\sum_{j=1}^n\mdj| \leq C \left(  \sum_{i,j=1}^n \langle (T-\Td)\phi_i, \phi'_j \rangle +\sup_{\substack{x\in E(X)\\\|x\|=1}}\fullnorm{(T-\Td)x}\sup_{\substack{y\in E'(X)\\\|y\|=1}}\fullnorm{(T'-\Td')y}\right),
  \end{equation*}
  and 
  \begin{equation*}
\max_{j=1,\dots, n}  |\mu - \mdj| \leq C \left(  \sum_{i,j=1}^n \langle (T-\Td)\phi_i, \phi'_j \rangle +\sup_{\substack{x\in E(X)\\\|x\|=1}}\fullnorm{(T-\Td)x}\sup_{\substack{y\in E'(X)\\\|y\|=1}}\fullnorm{(T'-\Td')y}\right)^{1/\alpha},
  \end{equation*}
  $\alpha$ being the ascent of the eigenvalue $\mu$, see \cite{Descloux1978b, Handbook2}.
\end{remark}
\subsubsection{Convergence of the eigenvalues for the SIP method}
We now restrict ourselves to the symmetric interior penalty method and consider
the fact that our operator is self-adjoint: then, $T'=T$, $\Td' = \Td$, and \eqref{eq:eigenvalue1} reads
\begin{equation}
  \label{eq:eigenvalue2}
  |\mu - \mdj| \leq C \left(  \sum_{i,j=1}^n \langle (T-\Td)\phi_i, \phi_j \rangle + \sup_{\substack{x\in E(X)\\\|x\|=1}}\fullnorm{(T-\Td)v}^2\right).
\end{equation}
At this stage, the goal is in bounding the first term at the right hand side of
the inequality by something quadratic in nature, to show that it converges as
fast as the second term. This is where the adjoint consistency of the SIP method is crucial.
Let $ y \in E(X)$ with $\|y\|=1$ and let $\psi\in X$ be the solution to
the adjoint problem
\begin{equation}
  \label{eq:adjoint}
\begin{aligned}
  L\psi  &= y \text{ in }\Omega\\
  B\psi &= 0\text{ on }\partial \Omega.
\end{aligned}
\end{equation}
This implies $Ty = \psi$, hence
\begin{equation}
  \label{eq:adjoint-appx}
    \inf_{\vd\in\Xd}\fullnorm{\psi -\vd} \leq C \sup_{\substack{x\in E(X)\\\|x\|=1}} \inf_{\vd\in\Xd}\fullnorm{x-\vd}. 
  \end{equation}
Consider then $x,y\in E(X)$, with $\|x\|=\|y\|=1$:
\begin{align*}
  \langle (T-\Td)x, y \rangle
  & = \langle (T-\Td)x, L\psi \rangle\\
  & = \ad((T-\Td)x, \psi)\\
  & = \ad((T-\Td)x, \psi- \vd).
\end{align*}
Finally, by the continuity of the bilinear form, the quasi optimality of the
discontinuous Galerkin method, and using \eqref{eq:adjoint-appx}, we
conclude that
\begin{equation*}
\begin{split}
  | \langle (T-\Td)x, y \rangle| 
  & \leq C \fullnorm{(T-\Td) x}\fullnorm{\psi-\vd}\\
  & \leq C \sup_{\substack{x\in E(X)\\\|x\|=1}} \inf_{\vd\in\Xd}\fullnorm{x-\vd}^2.
  \end{split}
\end{equation*}
Since clearly
\[\sup_{\substack{x\in E(X)\\\|x\|=1}} \fullnorm{(T-\Td)x}^2\leq C \sup_{\substack{x\in E(X)\\\|x\|=1}} \inf_{\vd\in\Xd}\fullnorm{x-\vd}^2,\]
from \eqref{eq:eigenvalue2} we conclude that 
\begin{equation*}
    \max_{j=1,\dots, n}|\mu - \mdj| \leq C\sup_{\substack{x\in E(X)\\\|x\|=1}} \inf_{\vd\in\Xd}\fullnorm{x-\vd}^2.
\end{equation*}
Since for every eigenvalue $\mu$ of $T$, $1/\mu$ is an eigenvalue of
\eqref{eq:mainproblem}, we have proven the following theorem.
\begin{theorem}
  \label{th:convergence}
  Let $\lambda$ be an eigenvalue of problem \eqref{eq:mainproblem} with
  associated eigenspace $U = \spn(u_1, \dots, u_n)$, with $\|u_i\| = 1$ for
  $i=1, \dots, n$. Then, there exist $n$ eigenvalue-eigenfunction
  pairs $\{(\ldj, \udj)\}_j$ of the
  finite dimensional problem \eqref{eq:discrete-prob}
    such that for all $j=1,\dots, n$
  \begin{align*}
    \min_{u\in U}\dgnorm{u-\udj} &\lesssim \sup_{u\in U}\inf_{\vd\in\Xd}\fullnorm{ u -\vd}\\
    |\lambda-\lambda_{\delta j}| &\lesssim \sup_{u\in U}\inf_{\vd\in\Xd}\fullnorm{ u -\vd}.
  \intertext{Furthermore, if the numerical solutions are obtained with the SIP method,}
    |\lambda-\lambda_{\delta j}| &\lesssim \sup_{u\in U}\inf_{\vd\in\Xd}\fullnorm{ u -\vd}^2.
  \end{align*}
  Finally, there are no spurious numerical eigenvalues or eigenvectors.
\end{theorem}
Given the approximation properties of the $\hp{}$ method and considering that
all eigenfunctions of \eqref{eq:mainproblem} belong to the space
$\mathcal{J}^\varpi_\gamma(\Omega)$ for a $\gamma > d/2$, we can also provide
the following corollary.
\begin{corollary}
 Let $\lambda$, $u$, $U$, $\ldj$, and $\udj$ be defined as in Theorem
 \ref{th:convergence} and let $N = \dim(\Xd)$. Then, there exist $C, b>0$ such
 that for all $u\in U$, for all $j=1,\dots, n$
   \begin{align*}
    \dgnorm{u-\udj} &\leq C e^{-b N^{1/(d+1)}}\\
    |\lambda-\lambda_{\delta j}| &\leq C e^{-b N^{1/(d+1)}}
  \intertext{Furthermore, if the numerical solutions are obtained with the SIP method,}
    |\lambda-\lambda_{\delta j}| &\leq C e^{-2b N^{1/(d+1)}}.
  \end{align*}

\end{corollary}
\section{Numerical results}
\label{sec:numerical-results}
In this section, we perform some numerical experiments on the linear eigenvalue
problem of finding $(\lambda, u) \in \mathbb{R}\times H^1(\Omega)$ such that $\|
u\|_{L^2(\Omega) } =1$ and
\begin{equation}
  \label{eq:2d-num-prob}
  \begin{aligned}
    (-\Delta + V) u &= \lambda u \text{ in }\Omega\\
    u&= 0 \text{ on }\partial \Omega.
  \end{aligned}
\end{equation}
The domain $\Omega$ is the $d$-dimensional cube with unitary edge $(-1/2, 1/2)^d$,
and $V$ is a potential with a singularity at the origin that will be specified
in the single cases. Since no exact solution is available, every numerical
solution is compared with the solution obtained at a higher degree of refinement
than those presented.

In all cases, the mesh is isotropically and geometrically refined around the
origin, with a geometric refinement ratio $\sigma = 1/2$. All elements are
axiparallel $d$-dimensional cubes. This means that,
introducing the refinement layers $\Omega_j$, $j=1,\dots, \ell$, such that for all $K\in \Omega_j$,
\begin{equation*}
   \inf_{x\in K}\|x\|_{\ell^\infty} = \sigma^{j+1}\quad j = 1,\dots, \ell-1
\end{equation*}
we
have
\begin{equation*}
  |K| = h_K^d = \sigma^{(j+1)d}.
\end{equation*}
Furthermore, the elements in $\Omega_\ell$ have a vertex on the singularity.
The polynomial slope $\slope$, defined as the parameter such that for all
$\vd\in\Xd$, if an element $K\in\Omega_j$ then
\begin{equation*}
  \vd_{|_K} \in \mathbb{Q}_{p_j}(K),
\end{equation*}
with
\begin{equation*}
  p_j = p_0 + \lfloor \slope (\ell - j) \rfloor
\end{equation*}
is instead variable between experiments, and it is
one of the main parameters whose role in the approximation we investigate. The
base polynomial degree is fixed at $p_0=1$.

All the simulations are obtained with C++ code based on the library
\texttt{deal.II} \cite{dealII85}. Furthermore, we use PETSc \cite{petsc} for the
solution of algebraic linear systems, and SLEPc \cite{slepc} for the solution of the
algebraic eigenvalue problem. The actual methods used will vary between the two
and the three dimensional cases, and will be specified in the respective sections.
The boundary conditions are imposed weakly, as is customary in the framework of
discontinuous Galerkin methods, so no spurious eigenvalue is introduced, as
shown in 
Section \ref{sec:numerical-appx}.

The results we will show in the following concern the estimation of the
$\mathrm{DG}$, $L^2(\Omega)$ and $L^\infty(\Omega)$ norms of the error, and
of the difference between the computed and the ``exact'' eigenvalue.
Furthermore, we will try to estimate the constants $b_X$ such that
\begin{equation*}
  \| u - \ud \|_X \leq C_X\exp({-b_X N^{1/(d+1)}}),
\end{equation*}
for $X= \mathrm{DG}, L^2(\Omega), L^\infty(\Omega)$, and 
\begin{equation*}
  | \lambda - \ld| \leq C_\lambda \exp( -b_{\lambda}N^{1/(d+1)} ).
\end{equation*}
Here, $\ud\in \Xd$ (resp. $\ld\in\mathbb{R}$) is the numerical eigenfunction
(resp. eigenvalue) computed
with $\dim(\Xd) = N$ and $u$ (resp. $\lambda$) is the exact one.

We start by illustrating the results obtained in the framework of a two
dimensional approximation.
\subsection{Two dimensional case}
\label{sec:num-lin-2d}
We solve problem \eqref{eq:2d-num-prob} with $d=2$ on a mesh built as shown in
Figure \ref{fig:2d-mesh}. An example of a numerically computed eigenfunction is
shown in Figure \ref{subfig:2d-lin-sol}. We can see the combination of the effect
of the laplacian with homogeneous Dirichlet boundary conditions and of the potential.
The cusp introduced by the potential is partially hidden by the rest of the
solution; in Figure \ref{subfig:2d-lin-sol-proj}, where a close up of the solution over a line
is represented, we can see it more clearly.

We consider three different potentials, given by $V(x) = r^{-\alpha}$, with
$\alpha \in\{ 1/2, 1, 3/2\}$. Clearly, the bigger the exponent $\alpha$, the lower the
regularity of the exact solution. In particular, from the point of view of
classical Sobolev spaces, denoting $u_\alpha$ as the solution of
\begin{align*}
  (-\Delta + r^{-\alpha }) u_\alpha  &= \lambda_\alpha u_\alpha\text{ in }\Omega\\
  u_\alpha &= 0\text{ on }\dOmega,
\end{align*}
we have $u_\alpha \in H^{3-\alpha-\xi}(\Omega)$, for any $\xi>0$. In particular,
the problem with $\alpha=3/2$ roughly corresponds to a two dimensional elliptic
problem in a domain with a crack, see \cite{Costabel2002}.
When considering weighted Sobolev spaces, we have
\begin{equation}
  u_\alpha \in \mathcal{J}^\varpi_{3-\alpha-\xi}(\Omega),
\end{equation}
again for any $\xi>0$.

From the algebraic point of view, the eigenpairs are computed using a
Krylov-Schur method \cite{Stewart2002}. Furthermore, a shift and invert spectral
transformation is
used to precondition and speed up computations. Due to the relatively small size of the problems
we consider here, the linear system introduced by the shift and invert spectral
transformation is solved via an LU decomposition. When considering the problem
set in three dimensions, we will see how to deal with problems with more degrees
of freedom, where memory availability becomes a concern.

\begin{figure}
  \centering
  \includegraphics[width=.35\textwidth]{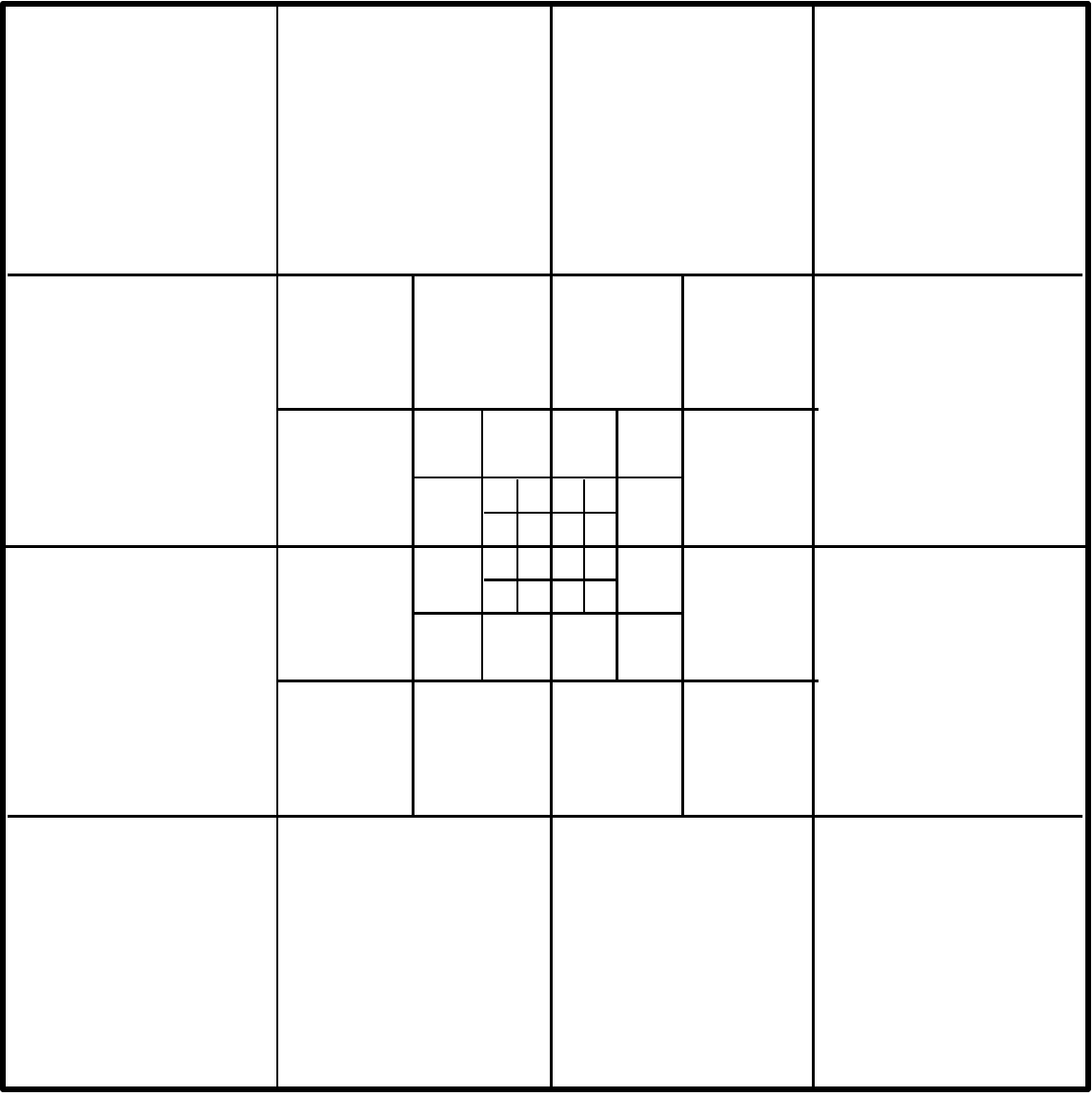}
  \caption{Example of a two dimensional mesh, with $\ell = 5$}
  \label{fig:2d-mesh}
\end{figure}
\begin{figure}
  \centering
  \begin{subfigure}[t]{.5\textwidth}
  \centering
  \includegraphics[width=.8\textwidth]{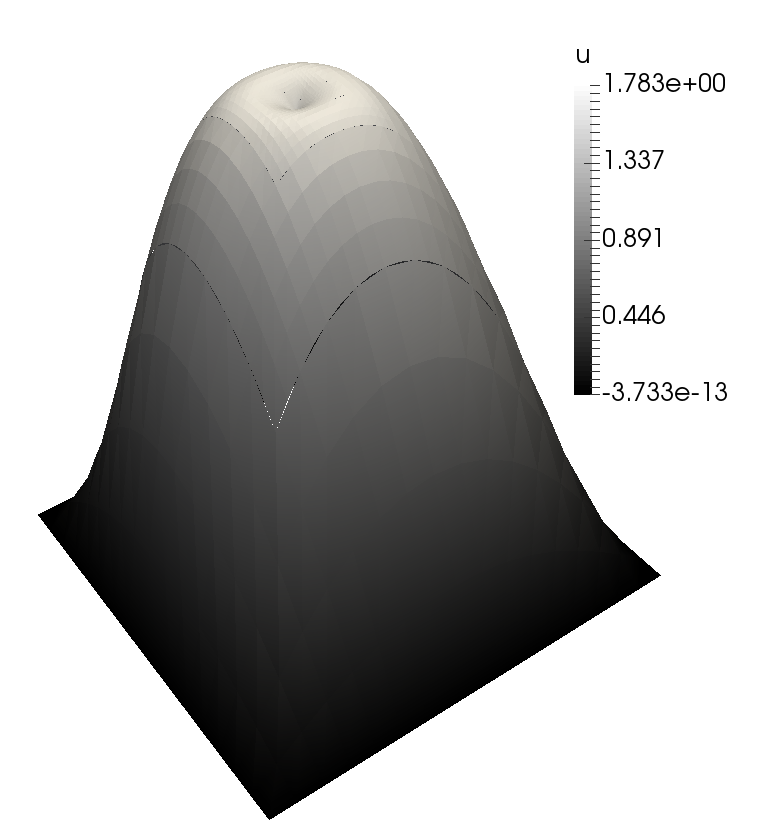}
  \caption{} \label{subfig:2d-lin-sol}
  \end{subfigure}\begin{subfigure}[t]{.5\textwidth}
  \centering
  \includegraphics[width=.7\textwidth]{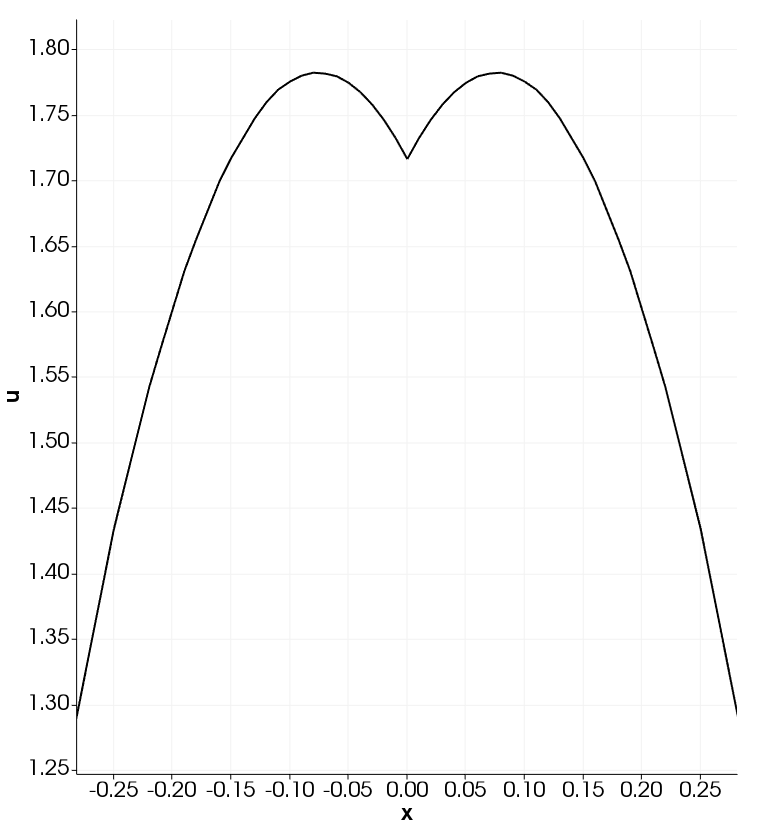}
  \caption{} \label{subfig:2d-lin-sol-proj}
  \end{subfigure}
  \caption{Numerical solution to \eqref{eq:2d-num-prob} with $V(x) = r^{-1}$. Figure \subref{subfig:2d-lin-sol}:
    representation vertically not to scale; the separation between some elements
  is an artifact of the visualization on grids with hanging nodes.
  Figure \subref{subfig:2d-lin-sol-proj}: close up around the singularity of the
  function $u(\cdot, 0)$, i.e., of $u$ on the line $\{y = 0\}$.}
\end{figure}
\begin{figure}
    \centering
    \begin{subfigure}{.5\textwidth}
    \includegraphics[width=\textwidth]{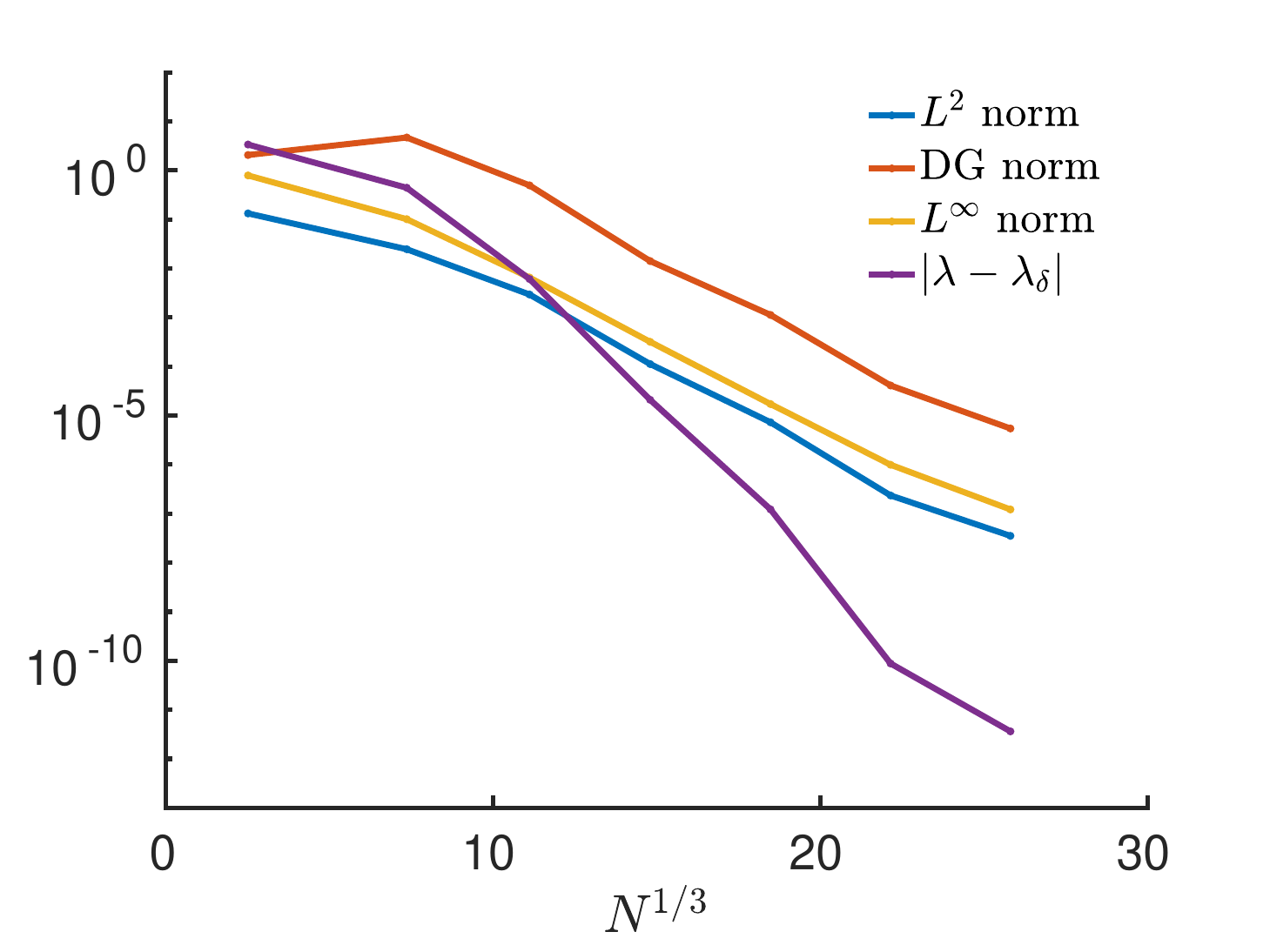}
    \caption{}\label{subfig:2d-lin-p050-s012}
    \end{subfigure}\begin{subfigure}{.5\textwidth}
    \includegraphics[width=\textwidth]{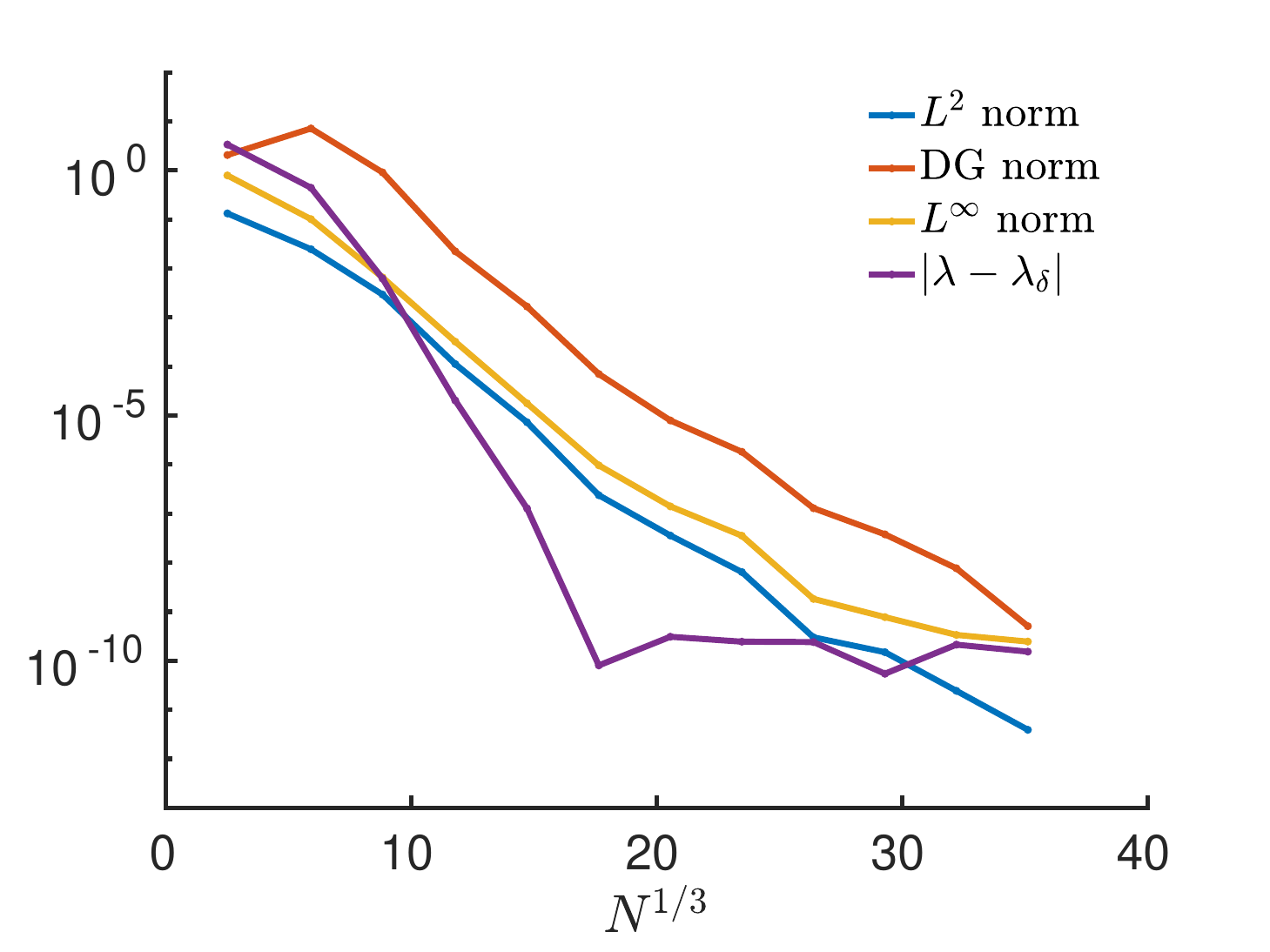}
    \caption{}\label{subfig:2d-lin-p050-s025}
    \end{subfigure}\\
    \begin{subfigure}{.5\textwidth}
    \includegraphics[width=\textwidth]{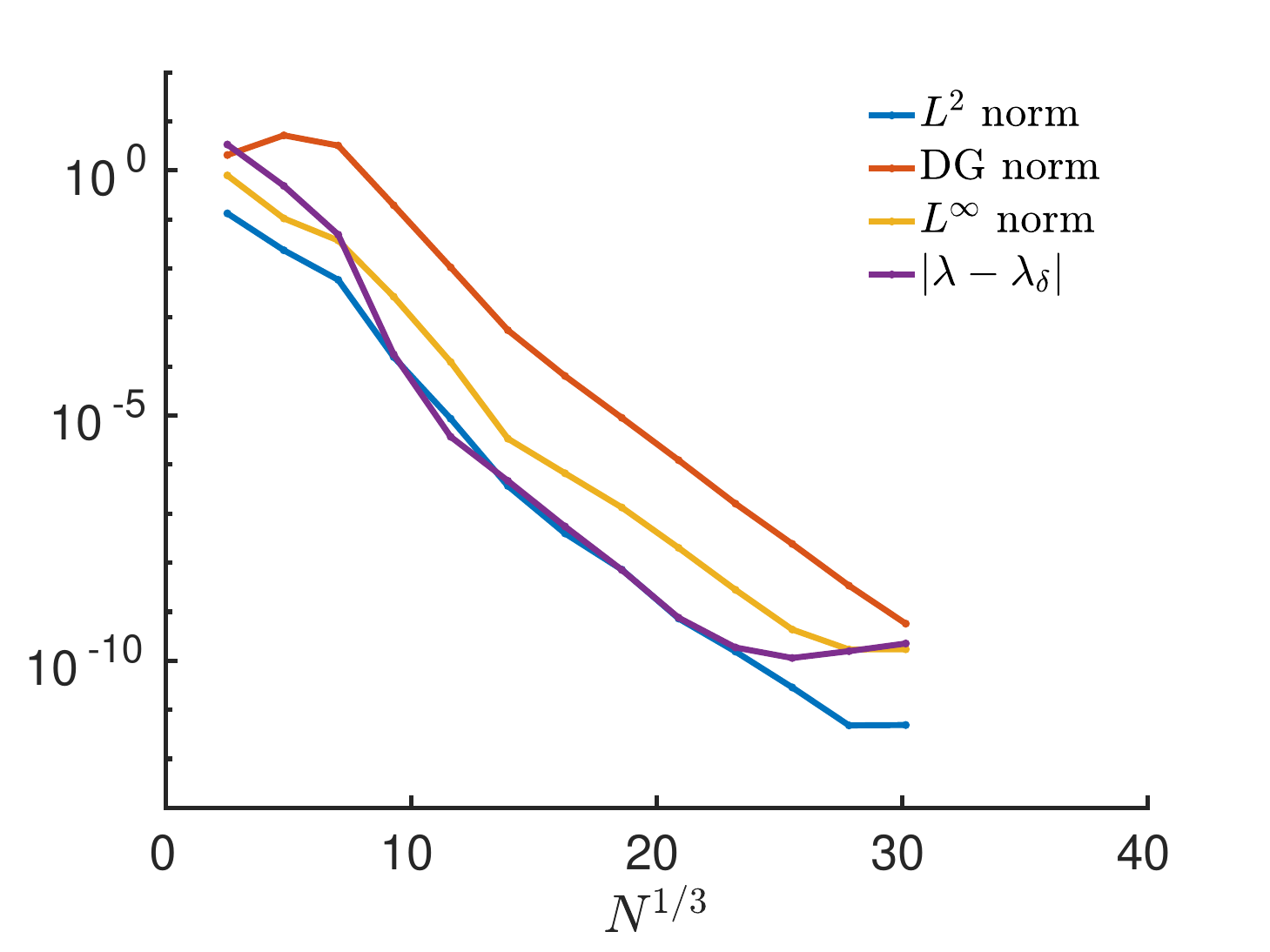}
    \caption{}\label{subfig:2d-lin-p050-s050}
    \end{subfigure}
    \caption{Errors for the numerical solution with potential $V(x) = r^{-1/2}$.
    Polynomial slope: $\slope = 1/8$ in Figure \subref{subfig:2d-lin-p050-s012};
    $\slope = 1/4$ in Figure \subref{subfig:2d-lin-p050-s025} and $\slope= 1/2$ in
  Figure \subref{subfig:2d-lin-p050-s050}.}\label{fig:2d-lin-p050}
\end{figure}
\begin{table}
\caption{Estimated coefficients. Potential: $r^{-1/2}$}
\label{table:2d-lin-p050-b}
\centering
\pgfplotstablevertcat{\output}{linear_results_pot-0_50_slope0_12_b} 
\pgfplotstablevertcat{\output}{linear_results_pot-0_50_slope0_25_b} 
\pgfplotstablevertcat{\output}{linear_results_pot-0_50_slope0_50_b} 
\pgfplotstabletypeset {\output}
\end{table}

\begin{figure}
    \centering
    \begin{subfigure}{.5\textwidth}
    \includegraphics[width=\textwidth]{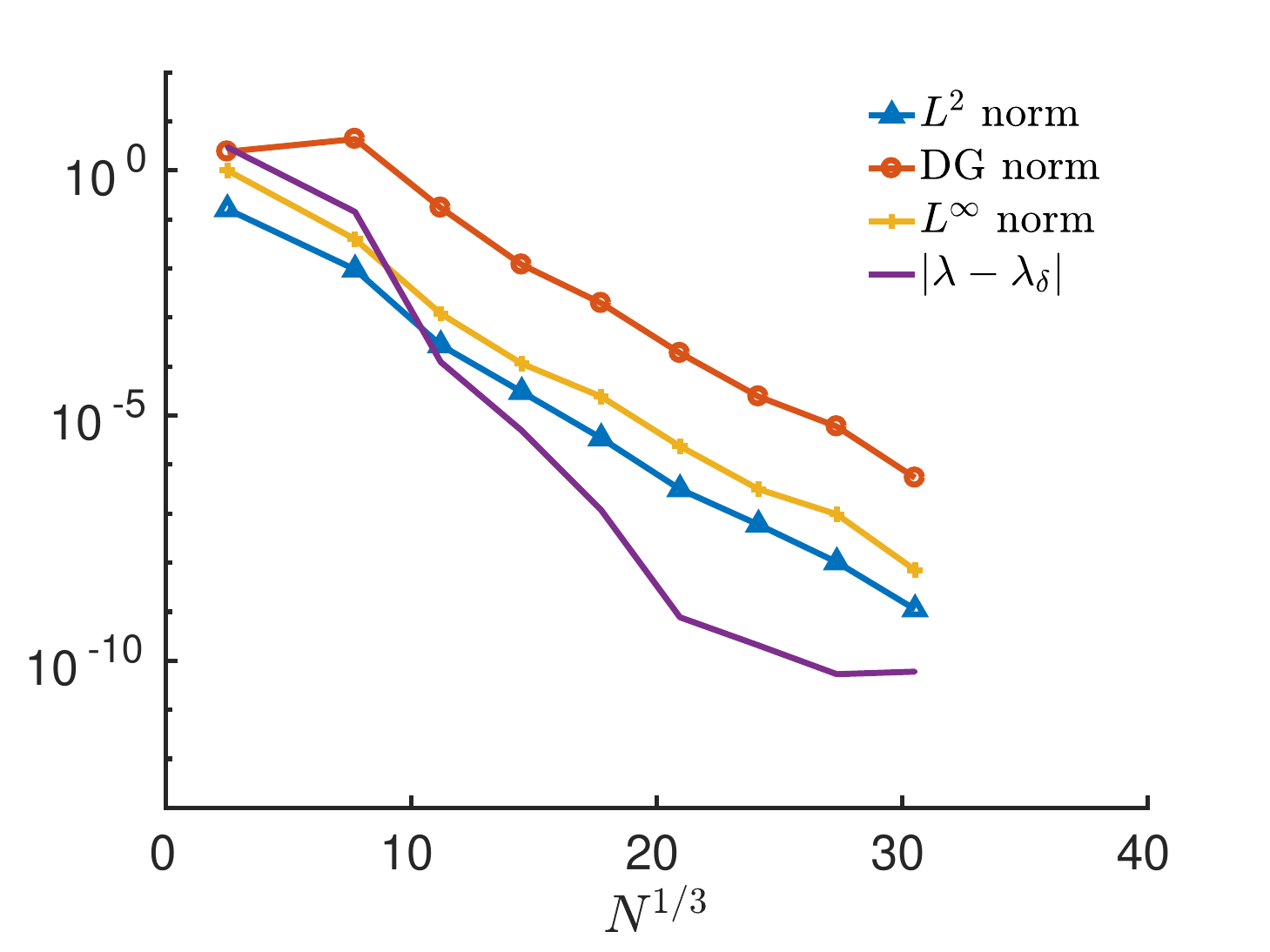}
    \caption{}\label{subfig:2d-lin-p100-s012}
    \end{subfigure}\begin{subfigure}{.5\textwidth}
    \includegraphics[width=\textwidth]{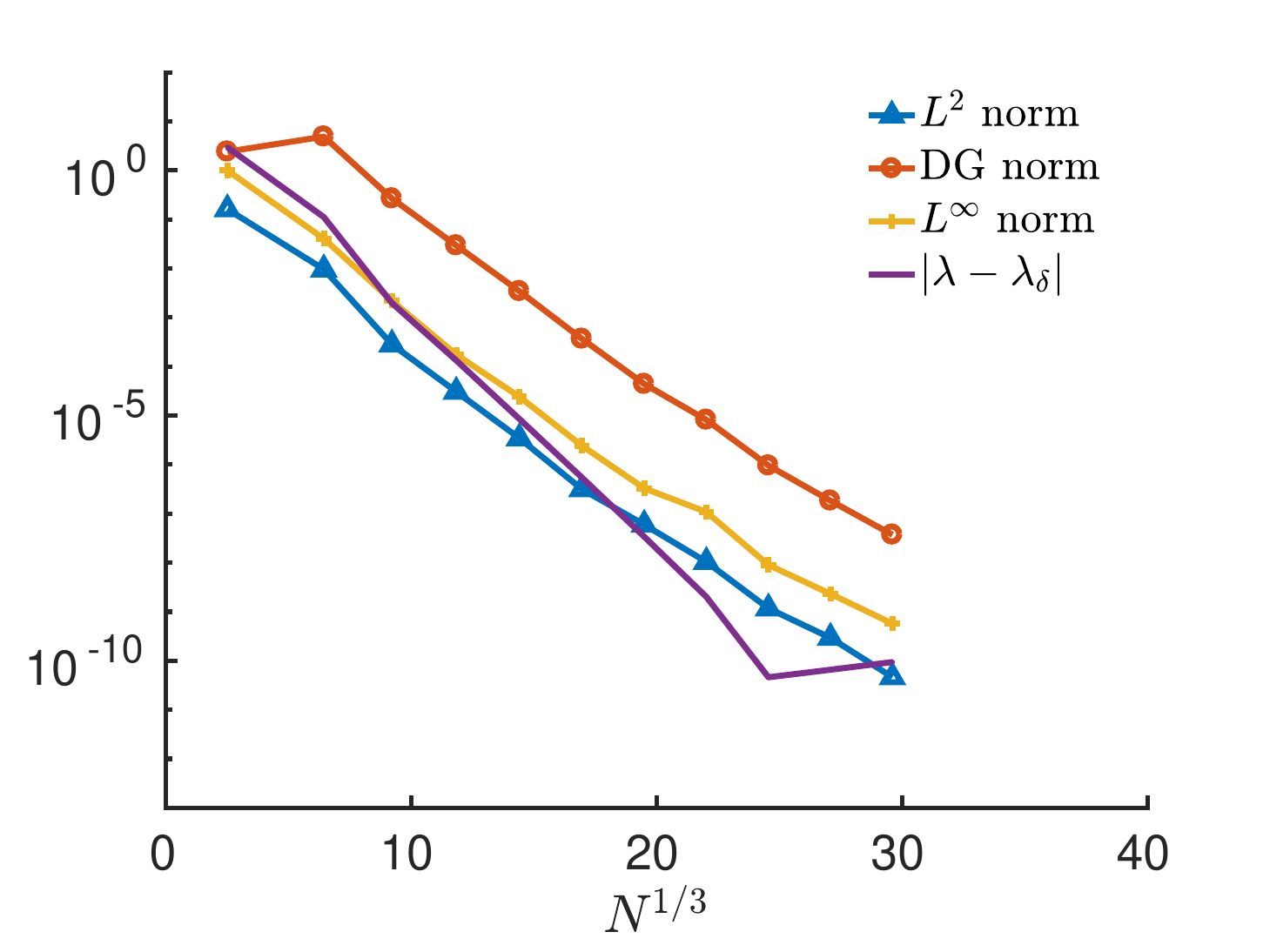}
    \caption{}\label{subfig:2d-lin-p100-s025}
    \end{subfigure}\\
    \begin{subfigure}{.5\textwidth}
    \includegraphics[width=\textwidth]{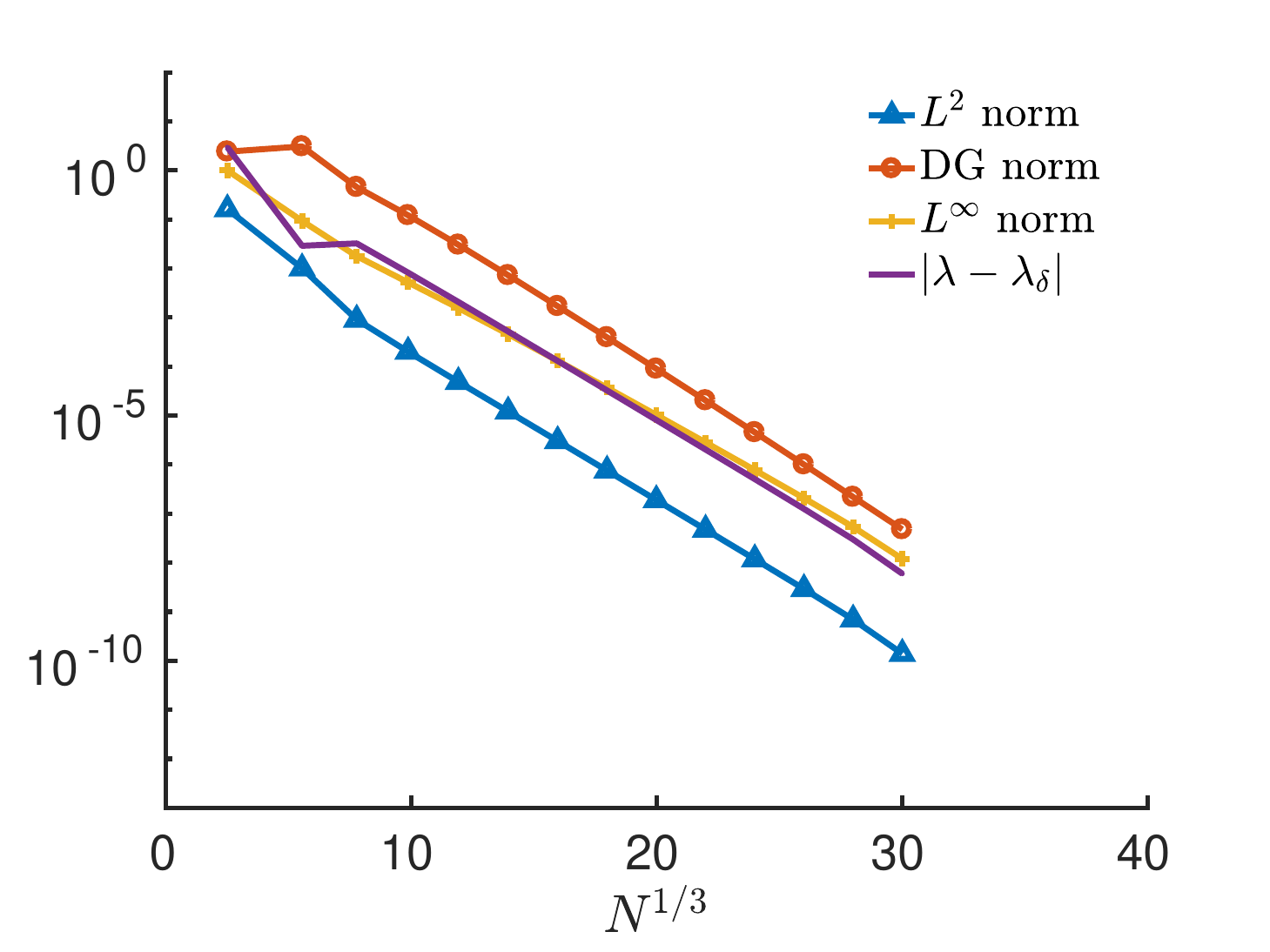}
    \caption{}\label{subfig:2d-lin-p100-s050}
    \end{subfigure}
    \caption{Errors for the numerical solution with potential $V(x) = r^{-1}$.
    Polynomial slope: $\slope = 1/8$ in Figure \subref{subfig:2d-lin-p100-s012};
    $\slope = 1/4$ in Figure \subref{subfig:2d-lin-p100-s025} and $\slope= 1/2$ in
  Figure \subref{subfig:2d-lin-p100-s050}.}\label{fig:2d-lin-p100}
\end{figure}
\begin{table}
\caption{Estimated coefficients. Potential: $r^{-1}$}
\label{table:2d-lin-p100-b}
\centering
\pgfplotstablevertcat{\output}{linear_results_pot-1_00_slope0_12_b} 
\pgfplotstablevertcat{\output}{linear_results_pot-1_00_slope0_25_b} 
\pgfplotstablevertcat{\output}{linear_results_pot-1_00_slope0_50_b} 
\pgfplotstabletypeset {\output}
\end{table}

\begin{figure}
    \centering
    \begin{subfigure}{.5\textwidth}
    \includegraphics[width=\textwidth]{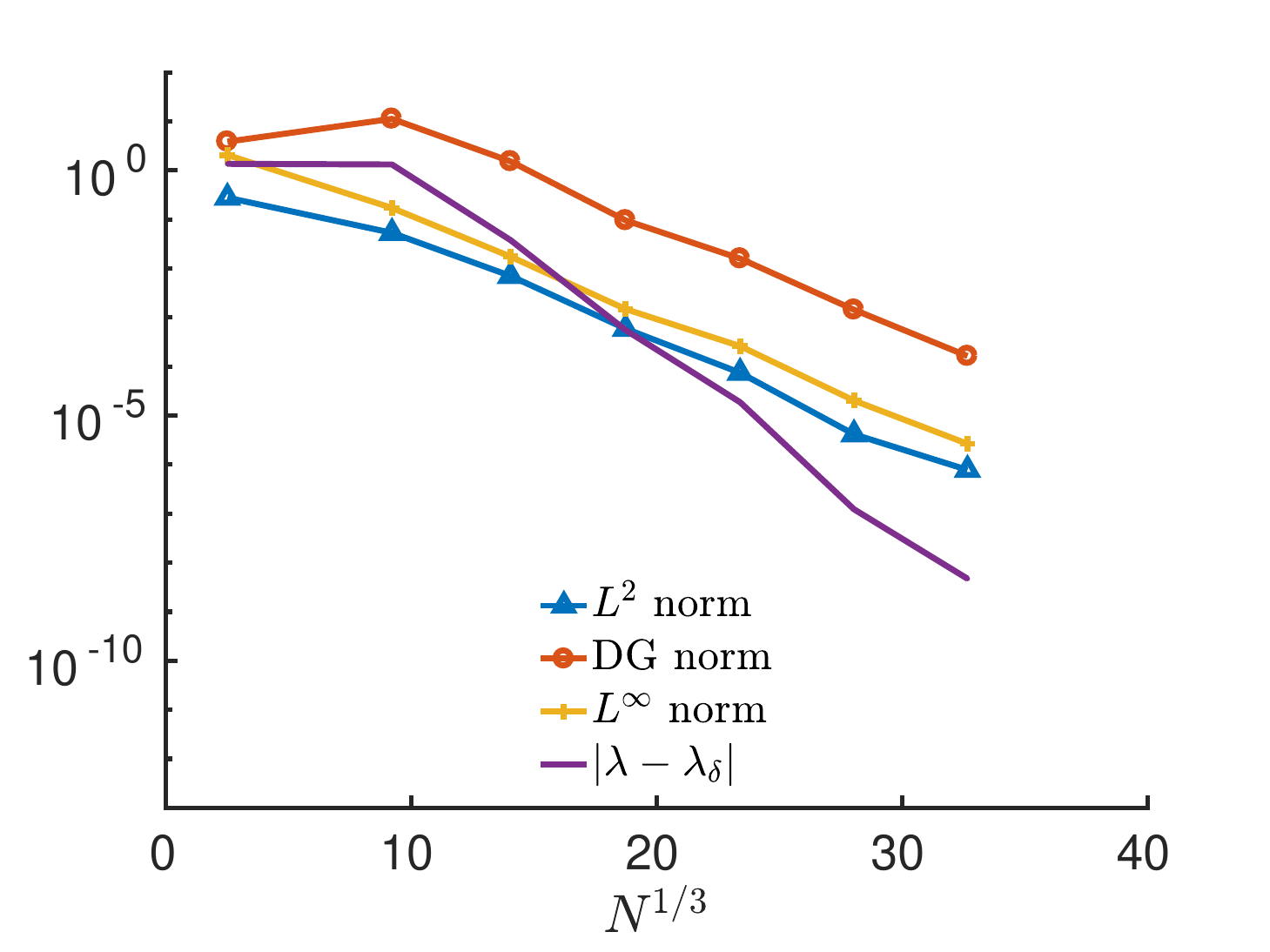}
    \caption{}\label{subfig:2d-lin-p150-s006}
    \end{subfigure}\begin{subfigure}{.5\textwidth}
    \includegraphics[width=\textwidth]{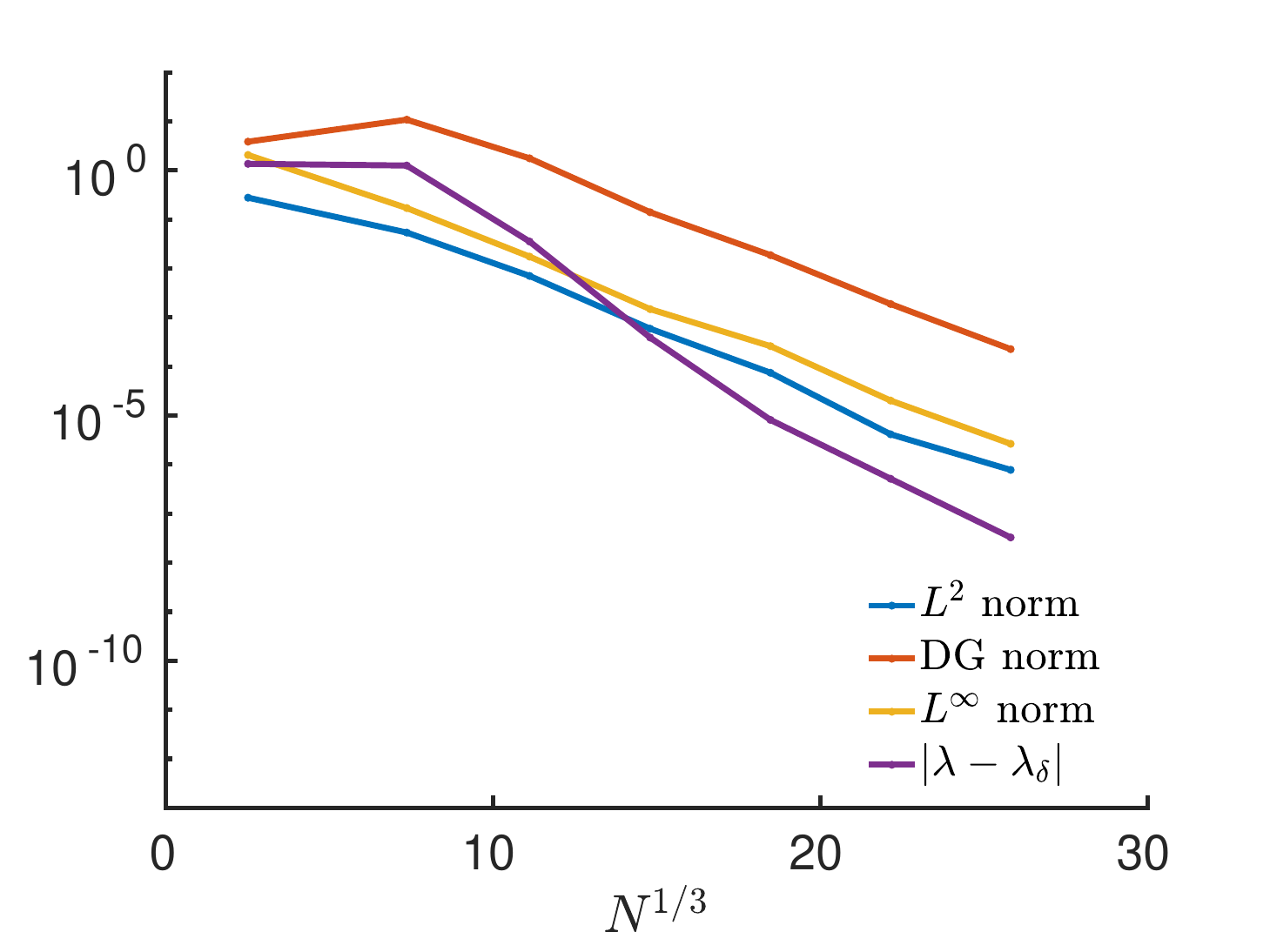}
    \caption{}\label{subfig:2d-lin-p150-s012}
    \end{subfigure}\\
    \begin{subfigure}{.5\textwidth}
    \includegraphics[width=\textwidth]{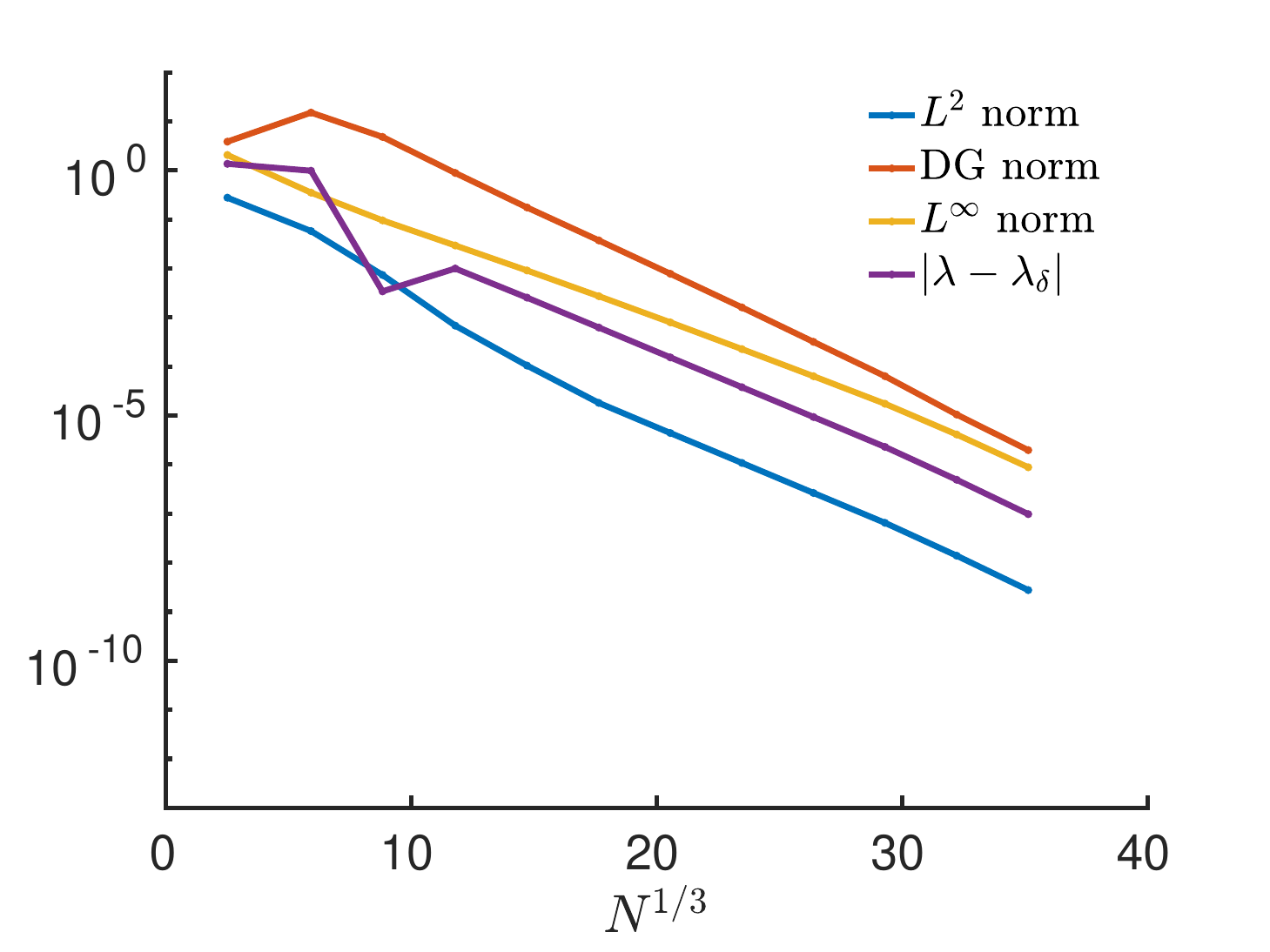}
    \caption{}\label{subfig:2d-lin-p150-s025}
    \end{subfigure}
    \caption{Errors for the numerical solution with potential $V(x) = r^{-3/2}$.
    Polynomial slope: $\slope = 1/16$ in Figure \subref{subfig:2d-lin-p150-s006};
    $\slope = 1/8$ in Figure \subref{subfig:2d-lin-p150-s012} and $\slope= 1/4$ in
  Figure \subref{subfig:2d-lin-p150-s025}.}\label{fig:2d-lin-p150}
\end{figure}
\begin{table}
\caption{Estimated coefficients. Potential: $r^{-3/2}$}
\centering
\pgfplotstablevertcat{\output}{linear_results_pot-1_50_slope0_06_b} 
\pgfplotstablevertcat{\output}{linear_results_pot-1_50_slope0_12_b} 
\pgfplotstablevertcat{\output}{linear_results_pot-1_50_slope0_25_b} 
\pgfplotstabletypeset {\output}
\label{table:2d-lin-p150-b}
\end{table}
\subsubsection{Analysis of the results}
The results on the error for the potential $V(x)  = r^{-1/2}$
are shown in Figure \ref{fig:2d-lin-p050}, and the estimated
coefficients are given in Table \ref{table:2d-lin-p050-b}.
Similarly, when the potential is given by $V(x) = r^{-1}$ the error curves are
in Figure \ref{fig:2d-lin-p100}, with coefficients $b_X$ in Table
\ref{table:2d-lin-p100-b}, and the case $V(x) = r^{-3/2}$ is reported in Figure \ref{fig:2d-lin-p150}
and Table \ref{table:2d-lin-p150-b}.

We can clearly see,
that in many cases the error reaches at some point a plateau; we estimate the coefficients
$b_X$ by linear regression on the points before the plateau. This will be done
for all subsequent potentials. Furthermore, as expected, the less regular the
potential, the slowest the convergence of the numerical solution.

Two phenomena are less expected from the point of view of the theory. The first
one is the emergence of a plateau at relatively high values compared to the
machine epsilon. Through the choice of different algebraic scheme, we can see
that we get a lower plateau: this is an indication that the dominating error at
the points where it is not converging to zero is the algebraic one. The fact
that matrices arising from the \hp{} method are ill conditioned explains the
size of the algebraic error. In practical applications, the fact that a relative
error of approximately $10^{-12}$ can be reached should be sufficient.

The second ``unexpected phenomenon'' is evident when looking at Figures
\ref{subfig:2d-lin-p050-s050}, \ref{subfig:2d-lin-p100-s025},
\ref{subfig:2d-lin-p150-s012}, and \ref{subfig:2d-lin-p150-s025}. We remark
that, after an initial part where the eigenvalue converges faster than the other
norms of the error, its rate of convergence then stabilizes to the same rate of
the other norms. This can be shown \cite{Cances2010} to be dependent on the quadrature formula
employed. When using a higher degree quadrature formula, the highest rate for
the eigenvalue error is recovered, see Figure \ref{fig:2d-lin-hq-p100-s025} and
Table \ref{table:2d-lin-hq-p100-b}, obtained with a higher quadrature formula and compare them with Figure \ref{subfig:2d-lin-p100-s025}
and Table \ref{table:2d-lin-p100-b}.
As a side effect of a higher quadrature order, the plateau is raised.

In
practice, one has to quite carefully balance computational cost, conditioning of
the matrix, and speed of convergence. The usefulness of this numerical
experiments lies therefore not only in the fact that we verify our theoretical
results and we see the impact of components of the error we did not account for
in the theoretical analysis, but also in the fact that we see, practically, how
the parameters affect the simulation for different exact solutions. Since by
asymptotic analysis we can see, locally and \emph{a priori}, how the solution of
a problem behaves, this gives an indication on how to construct and locally
\emph{a priori} optimize the \hp{} spaces.

\begin{figure}
  \centering
  \includegraphics[width = .5\textwidth]{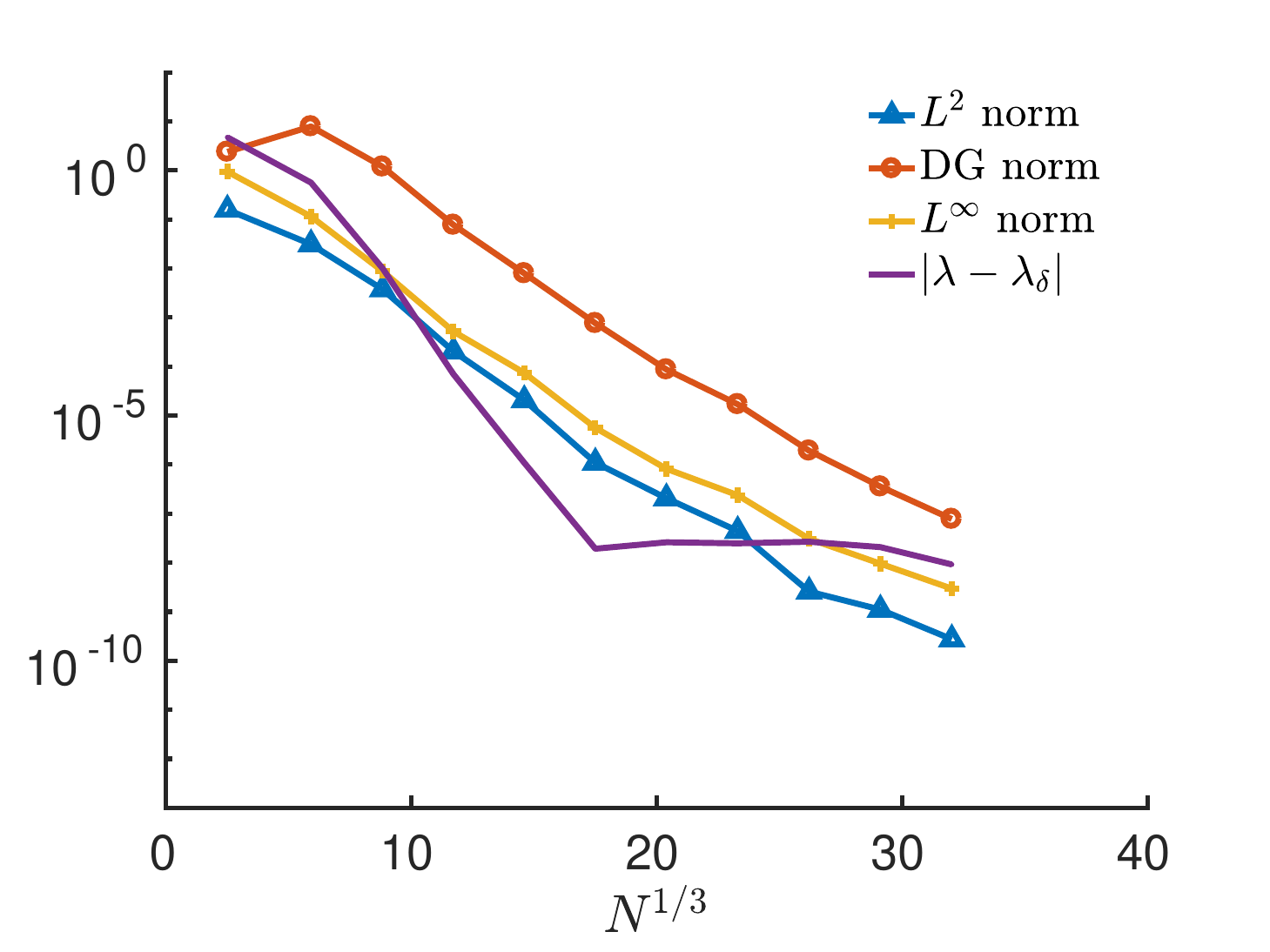}
  \caption{Errors of the numerical solution for $V(x) = r^{-1}$ and a high degree
    quadrature formula. Polynomial slope $\slope = 0.25$.}
  \label{fig:2d-lin-hq-p100-s025}
\end{figure}
\begin{table}
\caption{Estimated coefficients. Potential: $r^{-1}$, high degree quadrature formula}
\centering
\pgfplotstabletypeset {linear_highquad_results_pot-1_00_slope0_25_b}
\label{table:2d-lin-hq-p100-b}
\end{table}
\subsection{Three dimensional case}
In the three dimensional case, we replicate the setting introduced in Section
\ref{sec:num-lin-2d}. In this case, $\Omega = (-1/2, 1/2)^3$. Note that the
regularity of the solution of
\begin{align*}
  (-\Delta + r^{-\alpha }) u_\alpha  &= \lambda_\alpha u_\alpha\text{ in }\Omega\\
  u_\alpha &= 0\text{ on }\dOmega,
\end{align*}
scales differently with respect to $\alpha$, if compared to the two dimensional
case. Specifically, we have
\begin{equation*}
  u_\alpha \in H^{7/2-\alpha-\xi}(\Omega)
\end{equation*}
and
\begin{equation*}
  u_\alpha \in \mathcal{J}^\varpi_{7/2-\alpha-\xi}(\Omega),
\end{equation*}
for any $\xi>0$.
\begin{figure}
  \centering
  \includegraphics[width=.6\textwidth]{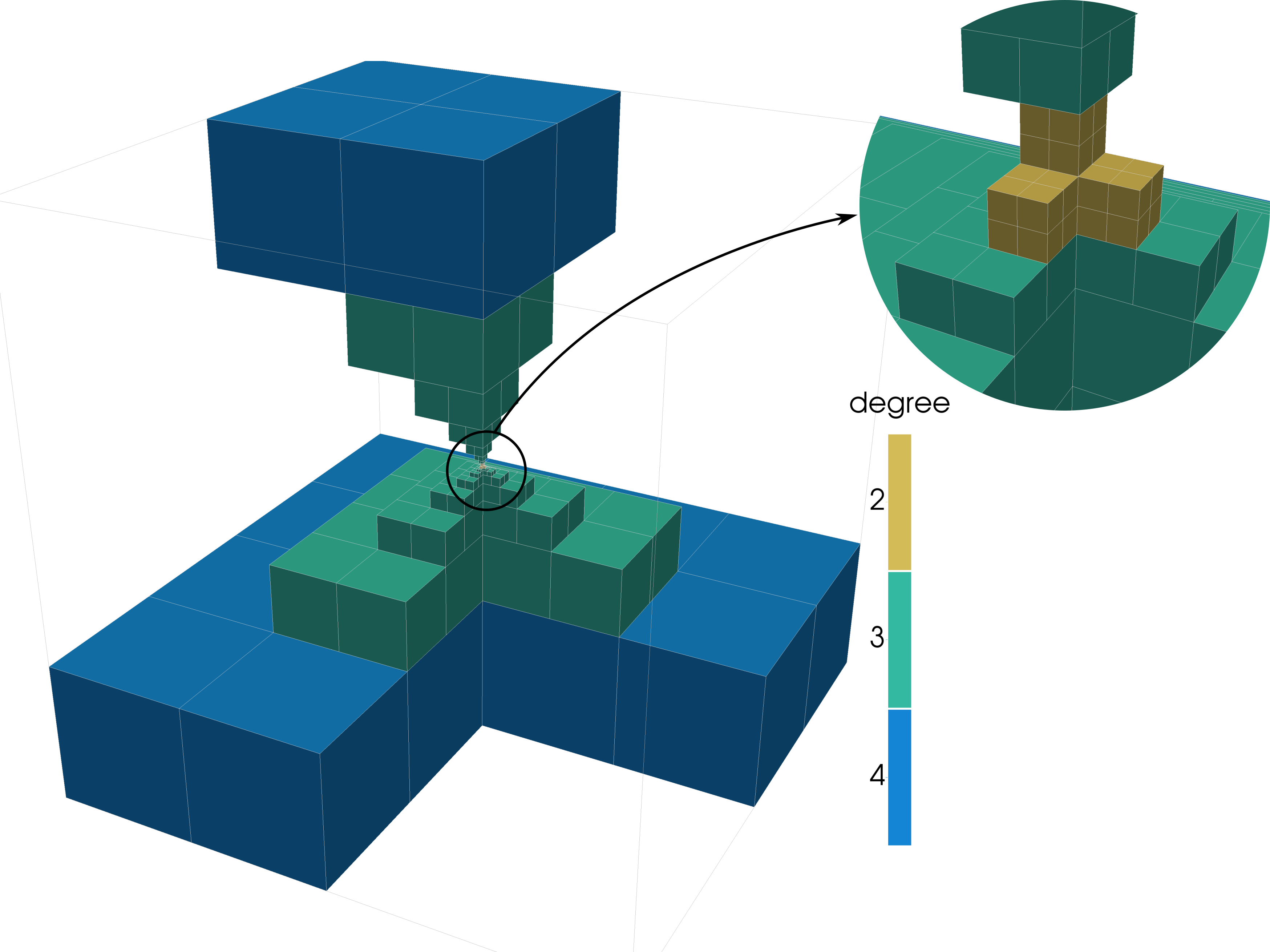}
  \caption{Example mesh for the three dimensional approximation}
  \label{fig:3d-mesh}
\end{figure}

\begin{figure}
  \centering
  \begin{subfigure}[t]{.5\textwidth}
  \centering
  \includegraphics[width=\textwidth]{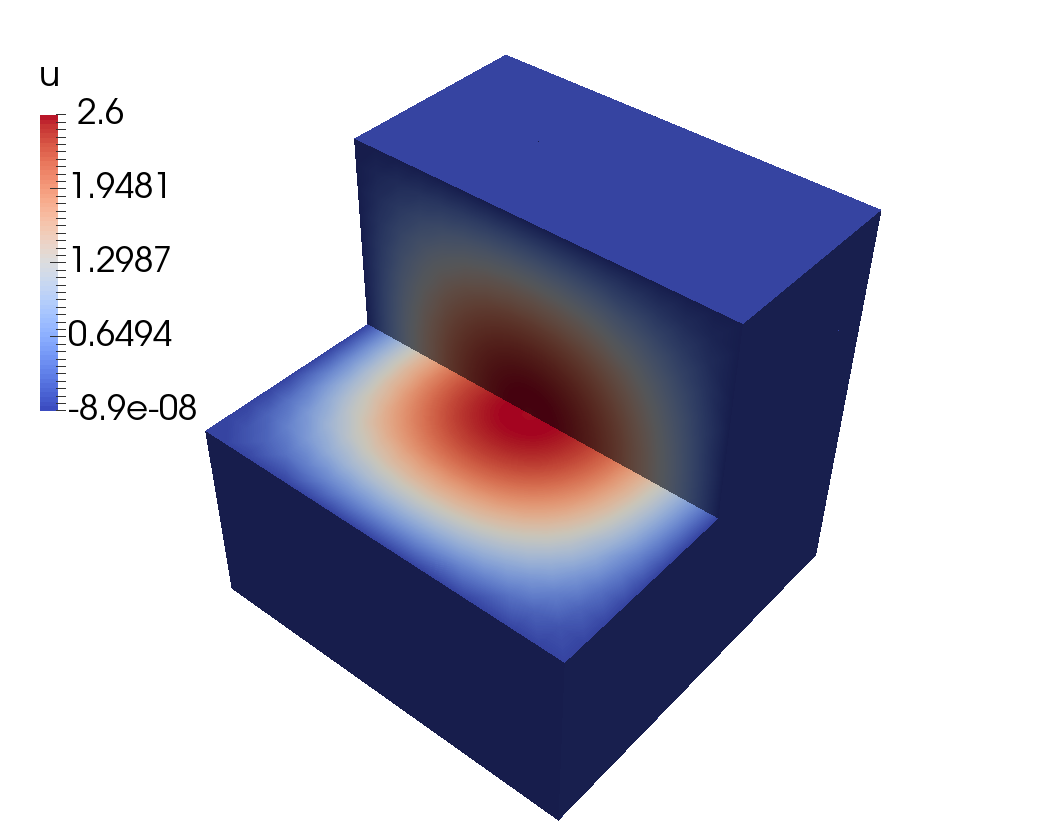}
    \end{subfigure}\begin{subfigure}[t]{.5\textwidth}
  \centering
  \includegraphics[width=.7\textwidth]{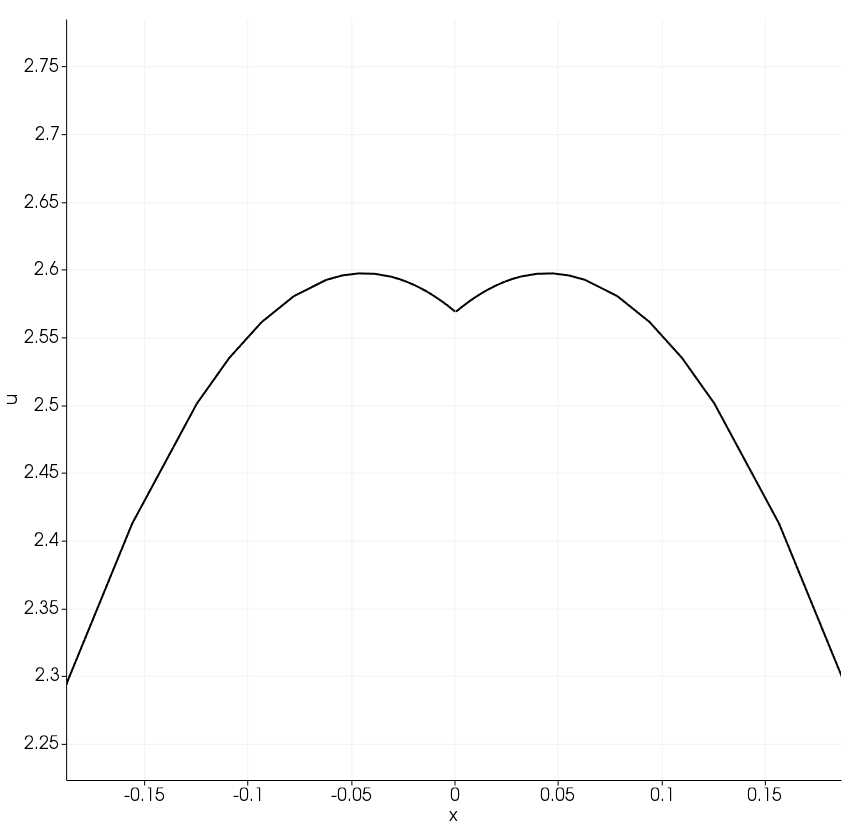}
    \end{subfigure}
  \caption{Numerical solution in the three dimensional case: solution in the
    cube, left, and close up near the origin of the restriction to the line $\{y= z= 0\}$, right}
  \label{fig:3d-lin-sol}
\end{figure}
The mesh is built in a tensor product way as in Section \ref{sec:num-lin-2d},
with refinement ratio $\sigma = 1/2$. A representation of a mesh is given in
Figure \ref{fig:3d-mesh}. The numerical solution for $V(x) = r^{-1}$ is shown
in Figure \ref{fig:3d-lin-sol}.

From the algebraic point of view, the assembled matrices are bigger in size and
less sparse, thus a direct LU method is less feasible than in the previous case
(up to completely unfeasible for the simulations with a high number of degrees
of freedom). Hence, we turn to iterative methods, and try to employ an
algebraic eigenvalue method that is not too sensible to the error introduced by
the linear solver. Therefore, the search for the eigenvalues is done with a
Jacobi-Davidson method \cite{Sleijpen1996}. 
Internally, we employ a biconjugate gradient stabilized method (BiCGS, \cite{VanderVorst1992,  Sleijpen1994})
as a linear solver, with simple Jacobi preconditioner. The tolerance for the
linear solver is set at $10^{-6}$, while the tolerance of the Jacobi-Davidson
method is set at $10^{-8}$.

\begin{figure}
    \centering
    \begin{subfigure}{.5\textwidth}
    \includegraphics[width=\textwidth]{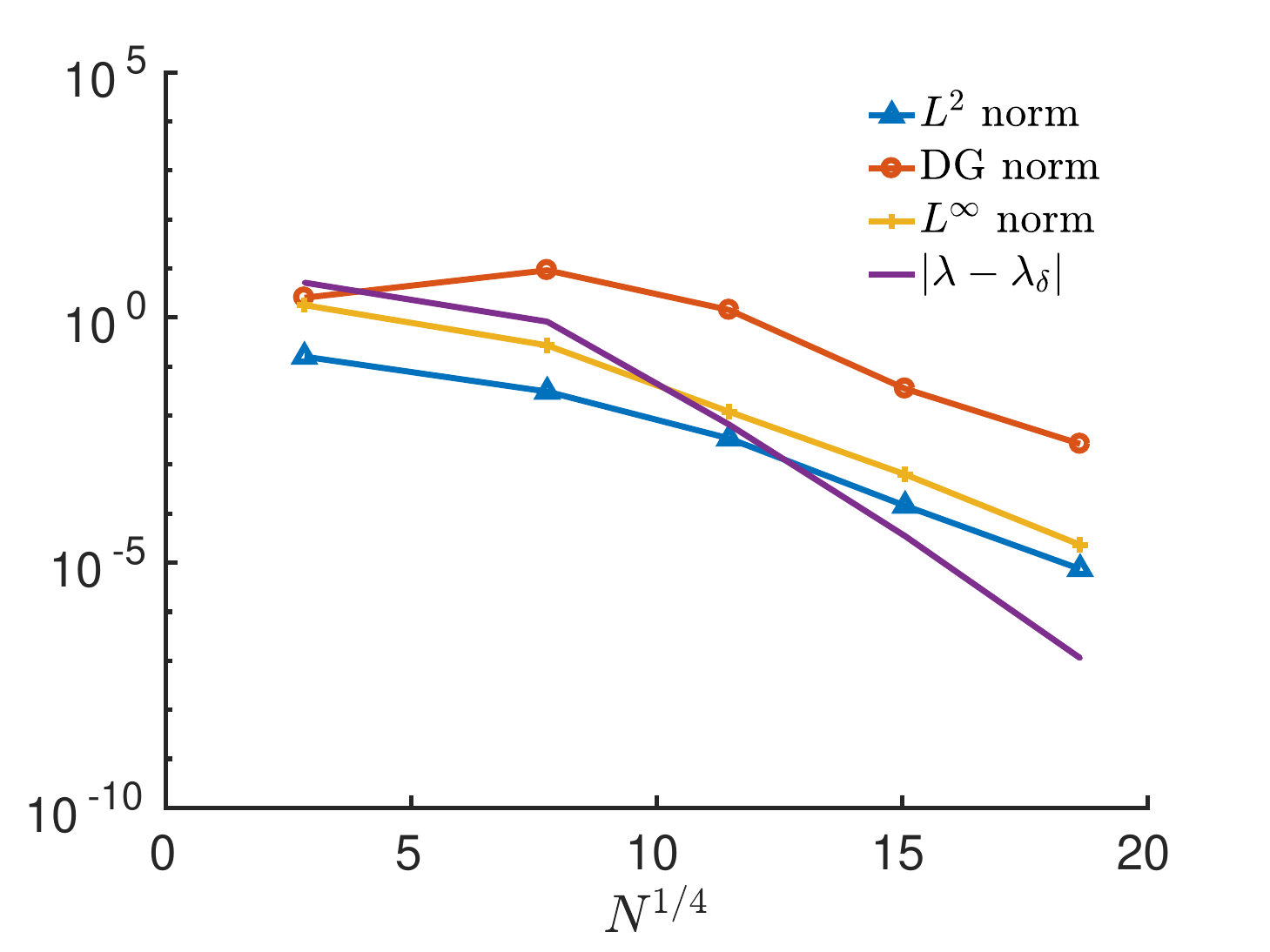}
        \end{subfigure}\begin{subfigure}{.5\textwidth}
    \includegraphics[width=\textwidth]{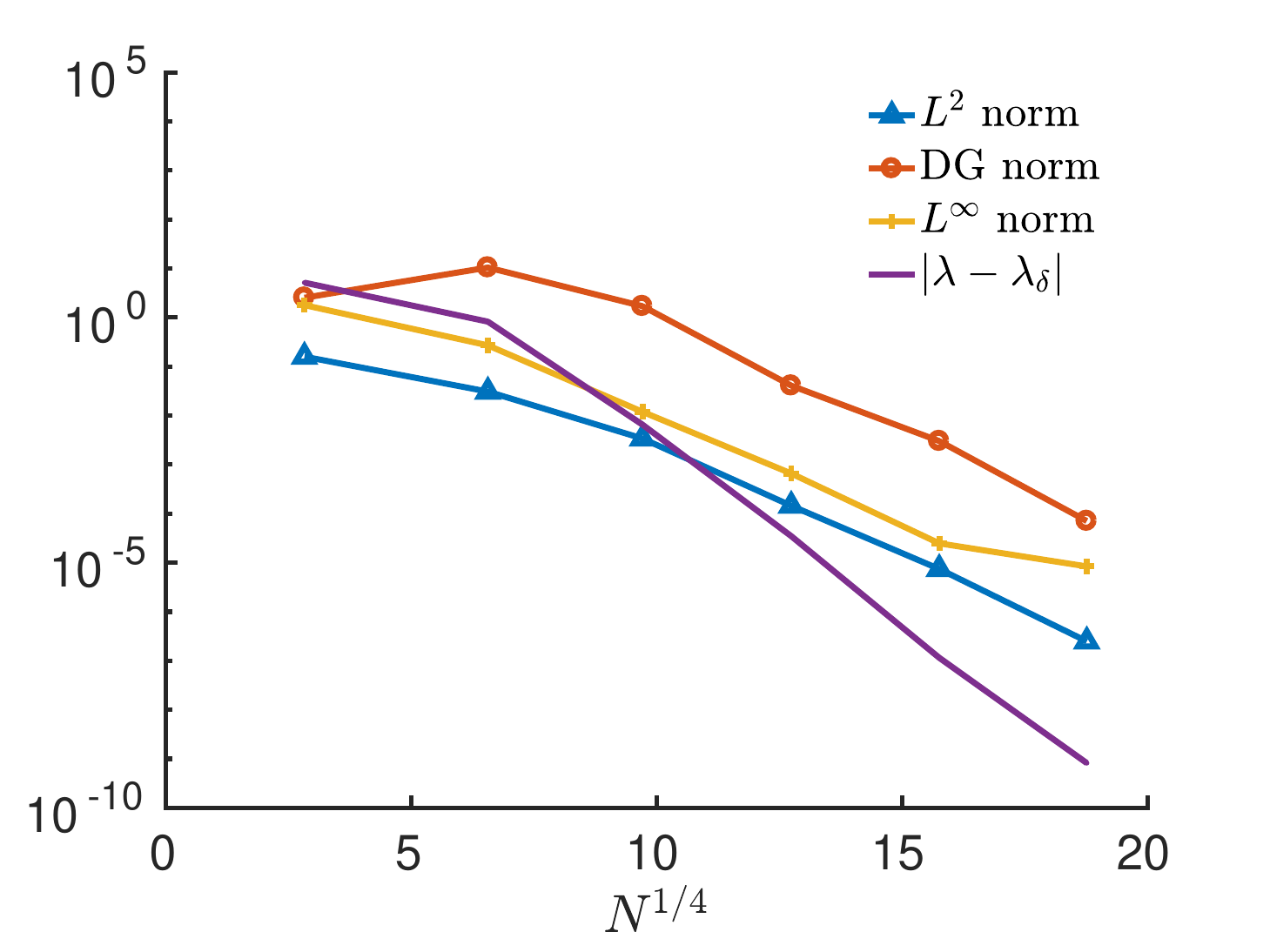}
      \end{subfigure}
  \caption{Errors of the numerical solution for $V(x) = r^{-1/2}$. Polynomial
    slope $\slope = 1/8$, left and $\slope = 1/4$, right.}
  \label{fig:3d-lin-p050}
\end{figure}
\begin{figure}
    \centering
    \begin{subfigure}{.5\textwidth}
    \includegraphics[width=\textwidth]{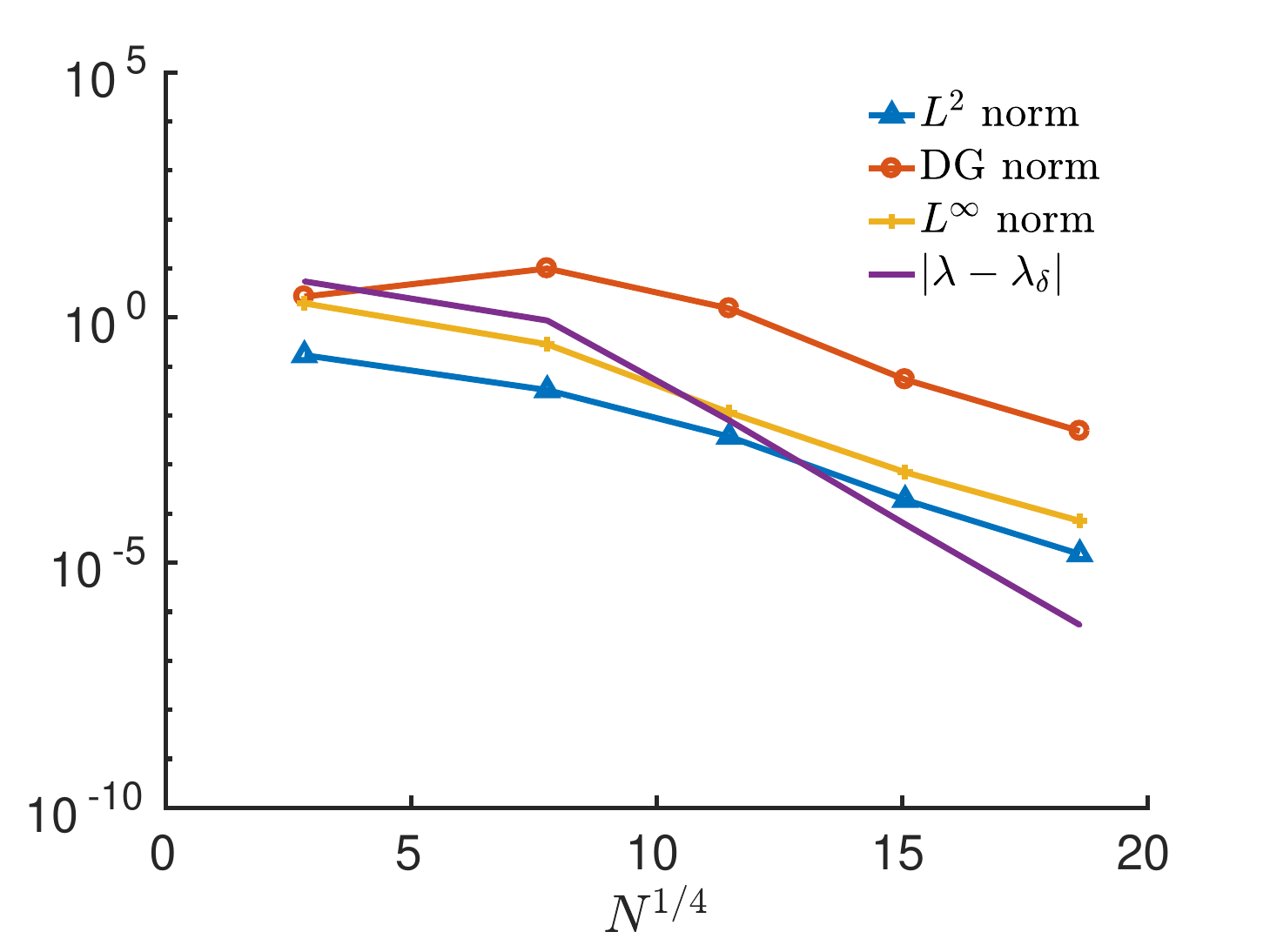}
        \end{subfigure}\begin{subfigure}{.5\textwidth}
    \includegraphics[width=\textwidth]{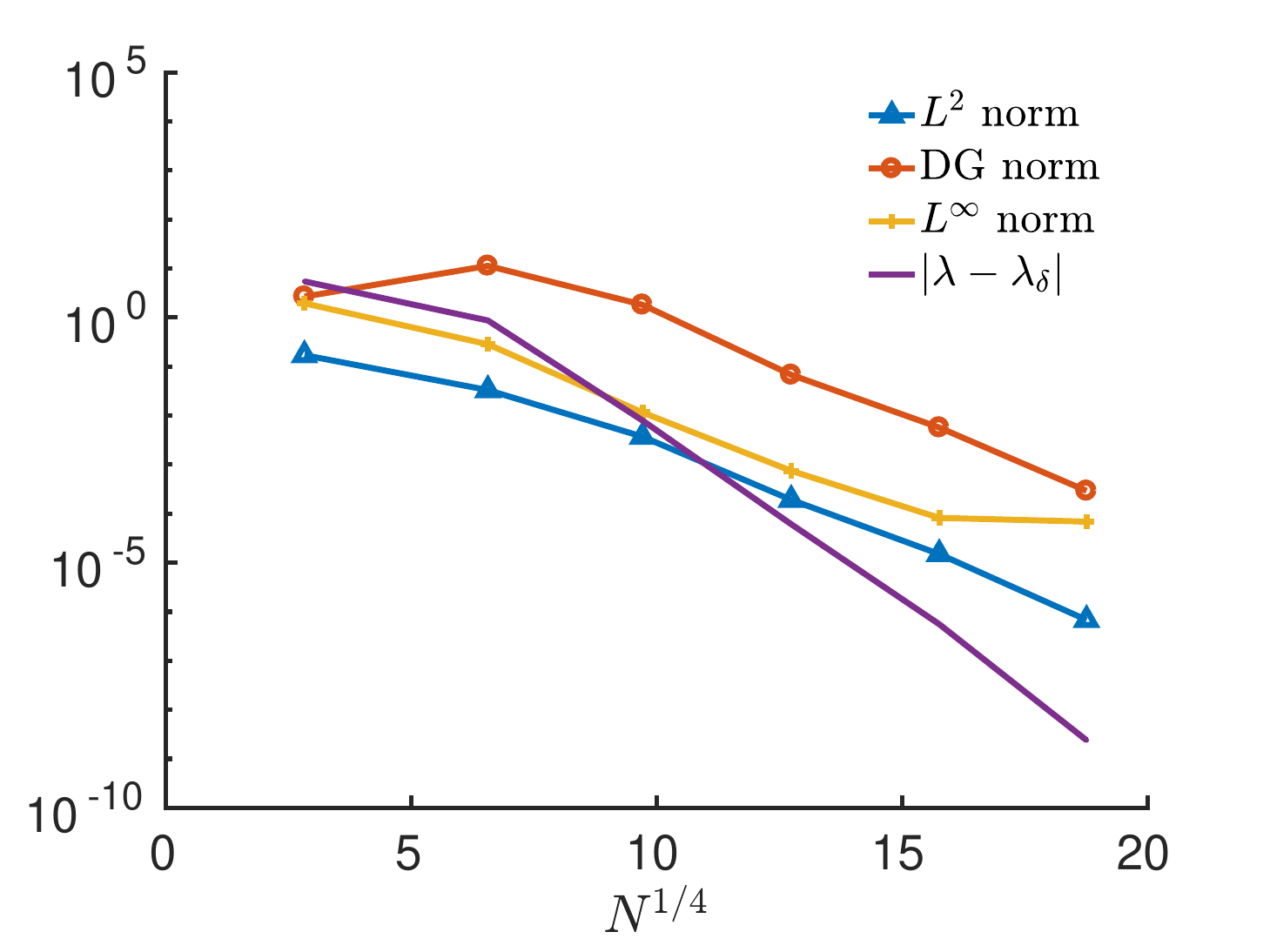}
      \end{subfigure}
    \caption{Errors of the numerical solution for $V(x) = r^{-1}$. Polynomial
    slope $\slope = 1/8$, left and $\slope = 1/4$, right.}
  \label{fig:3d-lin-p100}
\end{figure}
\begin{figure}
    \centering
    \begin{subfigure}{.5\textwidth}
    \includegraphics[width=\textwidth]{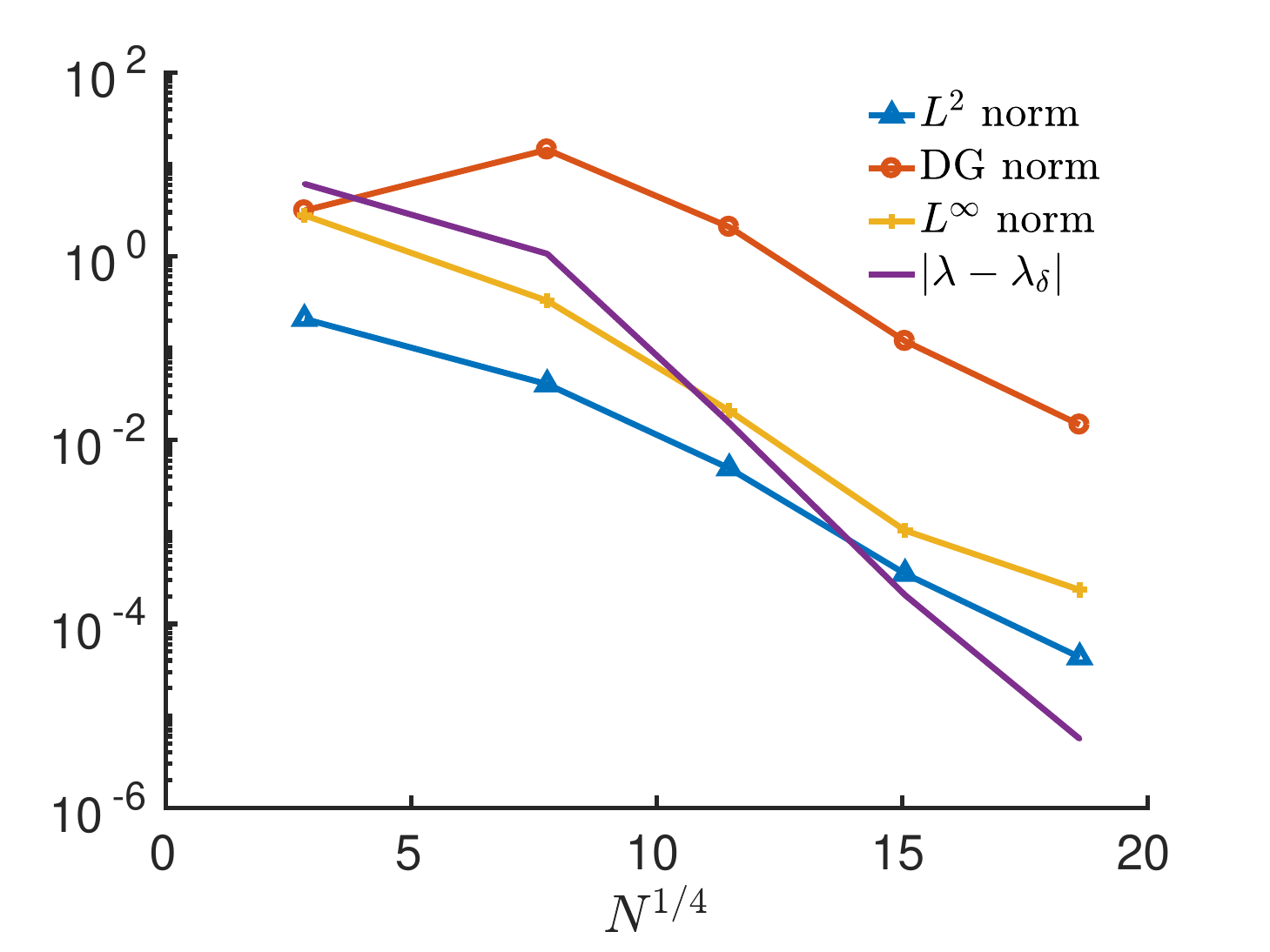}
        \end{subfigure}\begin{subfigure}{.5\textwidth}
    \includegraphics[width=\textwidth]{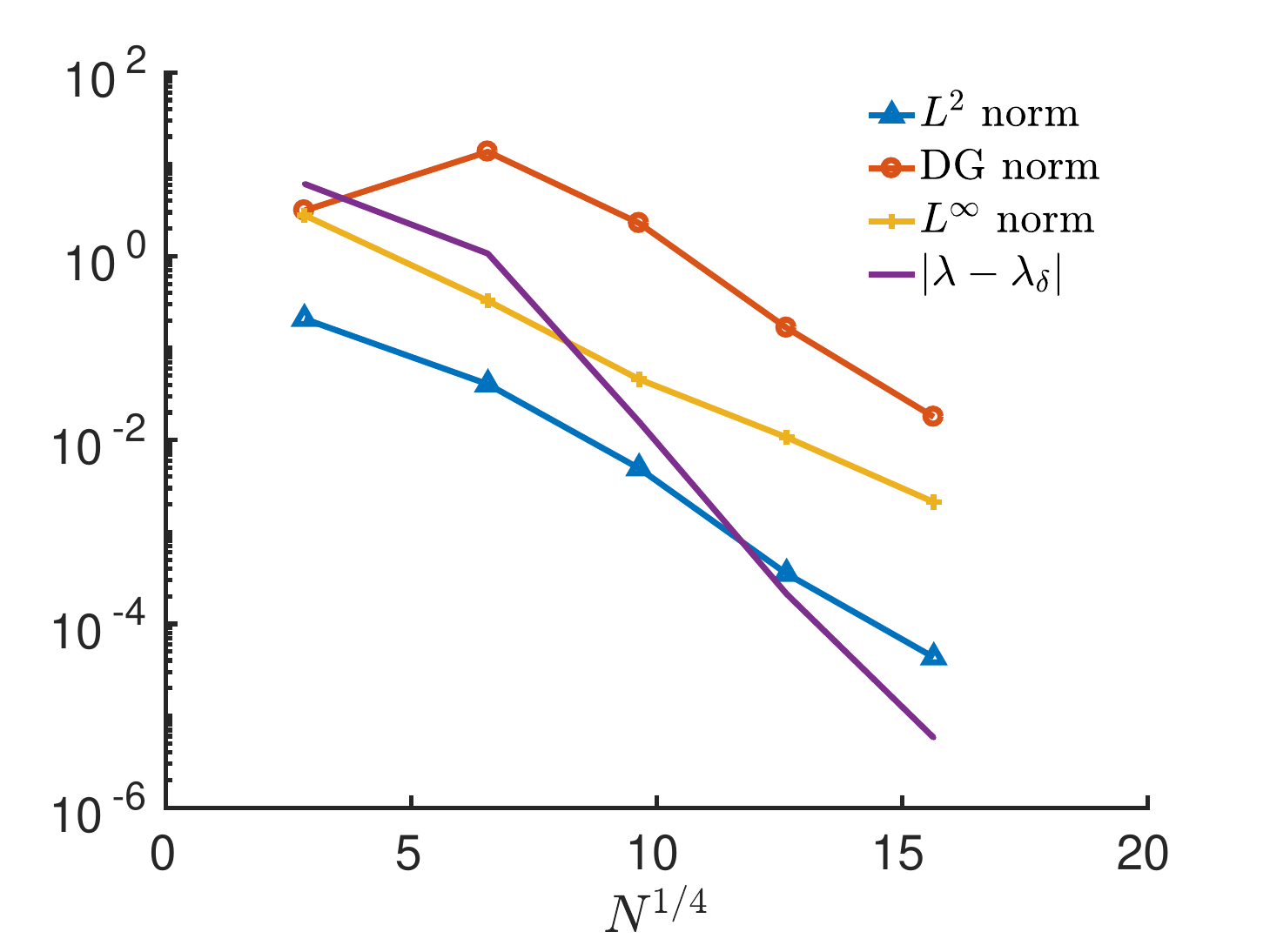}
      \end{subfigure}
    \caption{Errors of the numerical solution for $V(x) = r^{-3/2}$. Polynomial
    slope $\slope = 1/8$, left and $\slope = 1/4$, right.}
  \label{fig:3d-lin-p150}
\end{figure}

\begin{table}
\caption{Estimated coefficients. Potential: $r^{-1/2}$}
  \label{table:3d-lin-p050}
\centering
\pgfplotstablevertcat{\output}{linear3d_results_pot-0_50_slope0_12_b} 
\pgfplotstablevertcat{\output}{linear3d_results_pot-0_50_slope0_25_b} 
\pgfplotstabletypeset {\output}
\end{table}
\begin{table}
\caption{Estimated coefficients. Potential: $r^{-1}$}
  \label{table:3d-lin-p100}
\centering
\pgfplotstablevertcat{\output}{linear3d_results_pot-1_00_slope0_12_b} 
\pgfplotstablevertcat{\output}{linear3d_results_pot-1_00_slope0_25_b} 
\pgfplotstabletypeset {\output}
\end{table}
\begin{table}
\caption{Estimated coefficients. Potential: $r^{-3/2}$}
  \label{table:3d-lin-p150}
\centering
\pgfplotstablevertcat{\output}{linear3d_results_pot-1_50_slope0_12_b} 
\pgfplotstablevertcat{\output}{linear3d_results_pot-1_50_slope0_25_b} 
\pgfplotstabletypeset {\output}
\end{table}
\subsubsection{Analysis of the results}
Results for $V(x) = r^{-1/2}$ are given in Figure \ref{fig:3d-lin-p050} and
Table \ref{table:3d-lin-p050}, while the case $V(x) = r^{-1}$ is analyzed in Figure \ref{fig:3d-lin-p100} and
Table \ref{table:3d-lin-p100} and the errors and estimates when $V(x) =
r^{-3/2}$ are shown in Figure \ref{fig:3d-lin-p150} and Table
\ref{table:3d-lin-p150}. The three dimensional approximation has far more degrees of
freedom than the two dimensional one for a given level of refinement $\ell$, thus the
results we show have lower levels of refinement than the two dimensional ones.
This is partially balanced by the fact that the solutions are more regular, but
the errors are still obviously higher than those of the two dimensional case, at
the same number of degrees of freedom.
In the three dimensional case, we do not see a great effect neither of the
algebraic error nor of the quadrature formulas. The coefficients $b_\lambda$
listed in Tables \ref{table:3d-lin-p050} to \ref{table:3d-lin-p150} are almost the
double of the respective coefficients $b_\mathrm{DG}$; thus, if the effect of
the quadrature error is present, it is nonetheless negligible compared to
other sources of error for the quite comprehensive potentials and polynomial slopes considered in this
experiments.
\FloatBarrier
\bibliographystyle{amsalpha-abbrv}

\begin{thebibliography}{SvdVF94}

\bibitem[ABD{\etalchar{+}}17]{dealII85}
D.~Arndt, W.~Bangerth, D.~Davydov, T.~Heister, L.~Heltai, M.~Kronbichler,
  M.~Maier, J.-P. Pelteret, B.~Turcksin, and D.~Wells, \emph{The
  \texttt{deal.II} library, version 8.5}, Journal of Numerical Mathematics
  \textbf{25} (2017), no.~3, 137--145.

\bibitem[ABP06]{Antonietti2006}
P.~F. Antonietti, A.~Buffa, and I.~Perugia, \emph{{Discontinuous Galerkin
  approximation of the Laplace eigenproblem}}, Computer Methods in Applied
  Mechanics and Engineering \textbf{195} (2006), no.~25-28, 3483--3503.

\bibitem[Arn82]{Arnold1982}
D.~N. Arnold, \emph{{An Interior Penalty Finite Element Method with
  Discontinuous Elements}}, SIAM Journal on Numerical Analysis \textbf{19}
  (1982), no.~4, 742--760.

\bibitem[BAA{\etalchar{+}}17]{petsc}
S.~Balay, S.~Abhyankar, M.~F. Adams, J.~Brown, P.~Brune, K.~Buschelman,
  L.~Dalcin, V.~Eijkhout, W.~D. Gropp, D.~Kaushik, M.~G. Knepley, D.~A. May,
  L.~C. McInnes, K.~Rupp, B.~F. Smith, S.~Zampini, H.~Zhang, and H.~Zhang,
  \emph{{PETS}c {W}eb page}, \url{http://www.mcs.anl.gov/petsc}, 2017.

\bibitem[CCM10]{Cances2010}
E.~Canc{\`{e}}s, R.~Chakir, and Y.~Maday, \emph{{Numerical Analysis of
  Nonlinear Eigenvalue Problems}}, Journal of Scientific Computing \textbf{45}
  (2010), no.~1-3, 90--117.

\bibitem[CD02]{Costabel2002}
M.~Costabel and M.~Dauge, \emph{{Crack Singularities for General Elliptic
  Systems}}, Mathematische Nachrichten \textbf{235} (2002), no.~1, 29--49.

\bibitem[CDN10]{Costabel2010a}
M.~Costabel, M.~Dauge, and S.~Nicaise, \emph{{Mellin Analysis of Weighted
  Sobolev Spaces with Nonhomogeneous Norms on Cones}}, Around the Research of
  Vladimir Maz'ya I, Springer New York, 2010, pp.~105--136.

\bibitem[CDN12]{Costabel2012}
\bysame, \emph{{Analytic Regularity for Linear Elliptic Systems in Polygons and
  Polyhedra}}, Mathematical Models and Methods in Applied Sciences \textbf{22}
  (2012), no.~08, 1250015.

\bibitem[CDS05]{Costabel2005}
M.~Costabel, M.~Dauge, and C.~Schwab, \emph{{Exponential convergence of hp-FEM
  for Maxwell equations with weighted regularization in polygonal domains}},
  Mathematical Models and {\ldots} \textbf{15} (2005), no.~4, 575--622.

\bibitem[CL91]{Handbook2}
P.~G. Ciarlet and J.-L. Lions, \emph{Handbook of numerical analysis. {V}ol.
  {II}}, North-Holland, Amsterdam, 1991, Finite element methods. Part 1.
  \MR{1115235}

\bibitem[CR73]{Crouzeix1973}
M.~Crouzeix and P.-A. Raviart, \emph{{Conforming and nonconforming finite
  element methods for solving the stationary Stokes equations I}}, Revue
  fran{\c{c}}aise d'automatique informatique recherche op{\'{e}}rationnelle.
  Math{\'{e}}matique \textbf{7} (1973), no.~R3, 33--75.

\bibitem[CS98]{Cockburn1998}
B.~Cockburn and C.-W. Shu, \emph{{The Local Discontinuous Galerkin Method for
  Time-Dependent Convection-Diffusion Systems}}, SIAM Journal on Numerical
  Analysis \textbf{35} (1998), no.~6, 2440--2463.

\bibitem[DE12]{DiPietro2011}
D.~A. {Di Pietro} and A.~Ern, \emph{{Mathematical Aspects of Discontinuous
  Galerkin Methods}}, Mathématiques et Applications, vol.~69, Springer Berlin
  Heidelberg, Berlin, Heidelberg, 2012.

\bibitem[DNR78a]{Descloux1978a}
J.~Descloux, N.~Nassif, and J.~Rappaz, \emph{On spectral approximation. {I}.
  {T}he problem of convergence}, RAIRO Analyse Numérique \textbf{12} (1978),
  no.~2, 97--112, iii.

\bibitem[DNR78b]{Descloux1978b}
\bysame, \emph{On spectral approximation. {II}. {E}rror estimates for the
  {G}alerkin method}, RAIRO Analyse Numérique \textbf{12} (1978), no.~2,
  113--119, iii.

\bibitem[ES97]{Egorov1997}
Y.~V. Egorov and B.-W. Schulze, \emph{{Pseudo-Differential Operators,
  Singularities, Applications}}, Birkh{\"{a}}user Basel, Basel, 1997.

\bibitem[GB86a]{Gui1986a}
W.~Gui and I.~Babuška, \emph{{The h, p and h-p versions of the finite element
  method in 1 dimension. Part I. The Error Analysis of the p-Version}},
  Numerische Mathematik \textbf{612} (1986), 577--612.

\bibitem[GB86b]{Gui1986b}
\bysame, \emph{{The h, p and h-p versions of the finite element method in 1
  dimension. Part II. The Error analysis of the $h-$ and $h-p$ versions.}},
  Numerische Mathematik \textbf{49} (1986), no.~6, 613--657.

\bibitem[GB86c]{Gui1986c}
\bysame, \emph{{The h, p and h-p versions of the finite element method in 1
  dimension. Part III. The Adaptive h-p Version}}, Numerische Mathematik
  \textbf{683} (1986), 659--683.

\bibitem[GB86d]{Guo1986a}
B.~Guo and I.~Babuška, \emph{{The h-p version of the finite element method -
  Part 1: The basic approximation results}}, Computational Mechanics \textbf{1}
  (1986), no.~1, 21--41.

\bibitem[GB86e]{Guo1986b}
\bysame, \emph{{The h-p version of the finite element method - Part 2: General
  results and applications}}, Computational Mechanics \textbf{1} (1986), no.~3,
  203--220.

\bibitem[HRV05]{slepc}
V.~Hernandez, J.~E. Roman, and V.~Vidal, \emph{{SLEPc}: A scalable and flexible
  toolkit for the solution of eigenvalue problems}, ACM Transactions on
  Mathematical Software \textbf{31} (2005), no.~3, 351--362.

\bibitem[HW08]{Hesthaven2008}
J.~S. Hesthaven and T.~Warburton, \emph{{Nodal Discontinuous Galerkin
  Methods}}, Texts in Applied Mathematics, vol.~54, Springer New York, 2008.

\bibitem[KMR97]{Kozlov1997}
V.~Kozlov, V.~G. Maz'ya, and J.~Rossmann, \emph{{Elliptic boundary value
  problems in domains with point singularities}}, American Mathematical
  Society, 1997.

\bibitem[Kon67]{Kondratev1967}
V.~A. Kondrat'ev, \emph{Boundary value problems for elliptic equations in
  domains with conical or angular points}, Trudy Moskovskogo Matemati\v ceskogo
  Ob\v s\v cestva \textbf{16} (1967), 209--292. \MR{0226187}

\bibitem[MR10]{Mazya2010}
V.~G. Maz'ya and J.~Rossmann, \emph{{Elliptic Equations in Polyhedral
  Domains}}, Mathematical Surveys and Monographs, vol. 162, American
  Mathematical Society, apr 2010.

\bibitem[Nit72]{Nitsche1972}
J.~Nitsche, \emph{{On Dirichlet problems using subspaces with nearly zero
  boundary conditions}}, The Mathematical Foundations of the Finite Element
  Method with Applications to Partial Differential Equations, Elsevier, 1972,
  pp.~603--627.

\bibitem[RH73]{Reed1973}
W.~H. Reed and T.~Hill, \emph{Triangular mesh methods for the neutron transport
  equation}, Tech. report, Los Alamos Scientific Lab., (USA), 1973.

\bibitem[Riv08]{Riviere2008}
B.~Rivière, \emph{{Discontinuous Galerkin Methods for Solving Elliptic and
  Parabolic Equations}}, Society for Industrial and Applied Mathematics, jan
  2008.

\bibitem[SSW13a]{Schotzau2013b}
D.~Sch\"{o}tzau, C.~Schwab, and T.~P. Wihler, \emph{{$hp$-dGFEM for second
  order elliptic problems in polyhedra. II: Exponential convergence}}, SIAM
  Journal on Numerical Analysis \textbf{51} (2013), no.~4, 2005--2035.

\bibitem[SSW13b]{Schotzau2013a}
D.~Sch\"{o}tzau, C.~Schwab, and T.~Wihler, \emph{{$hp$-dGFEM for Second-Order
  Elliptic Problems in Polyhedra I: Stability on Geometric Meshes}}, SIAM
  Journal on Numerical Analysis \textbf{51} (2013), no.~3, 1610--1633.

\bibitem[SSW16]{Schotzau2016}
D.~Sch{\"{o}}tzau, C.~Schwab, and T.~P. Wihler, \emph{{$hp$-dGFEM for
  second-order mixed elliptic problems in polyhedra}}, Mathematics of
  Computation \textbf{85} (2016), no.~299, 1051--1083.

\bibitem[Ste02]{Stewart2002}
G.~W. Stewart, \emph{{A Krylov--Schur Algorithm for Large Eigenproblems}}, SIAM
  Journal on Matrix Analysis and Applications \textbf{23} (2002), no.~3,
  601--614.

\bibitem[SV96]{Sleijpen1996}
G.~L. Sleijpen and H.~A. {Van der Vorst}, \emph{{A Jacobi–Davidson Iteration
  Method for Linear Eigenvalue Problems}}, SIAM Journal on Matrix Analysis and
  Applications \textbf{17} (1996), no.~2, 401--425.

\bibitem[SvdVF94]{Sleijpen1994}
G.~L.~G. Sleijpen, H.~A. van~der Vorst, and D.~R. Fokkema, \emph{{BiCGstab(l)
  and other hybrid Bi-CG methods}}, Numerical Algorithms \textbf{7} (1994),
  no.~1, 75--109.

\bibitem[SW10]{Stamm2010}
B.~Stamm and T.~P. Wihler, \emph{{hp-optimal discontinuous Galerkin methods for
  linear elliptic problems}}, Mathematics of Computation \textbf{79} (2010),
  2117--2133.

\bibitem[vdV92]{VanderVorst1992}
H.~A. van~der Vorst, \emph{{Bi-CGSTAB: A Fast and Smoothly Converging Variant
  of Bi-CG for the Solution of Nonsymmetric Linear Systems}}, SIAM Journal on
  Scientific and Statistical Computing \textbf{13} (1992), no.~2, 631--644.

\bibitem[Whe78]{Wheeler1978}
M.~F. Wheeler, \emph{{An Elliptic Collocation-Finite Element Method with
  Interior Penalties}}, SIAM Journal on Numerical Analysis \textbf{15} (1978),
  no.~1, 152--161.

\end{thebibliography}
\newcommand{\etalchar}[1]{$^{#1}$}
\newcommand{\noopsort}[1]{}
\providecommand{\bysame}{\leavevmode\hbox to3em{\hrulefill}\thinspace}
\providecommand{\MR}{\relax\ifhmode\unskip\space\fi MR }
\providecommand{\MRhref}[2]{  \href{http://www.ams.org/mathscinet-getitem?mr=#1}{#2}
}
\providecommand{\href}[2]{#2}

\end{document}